\documentclass[12pt,oneside,english]{amsart}
\usepackage{amssymb,amsmath,latexsym,amsthm}

\usepackage[a4paper,top=3cm,bottom=3cm,left=3cm,right=3cm,marginparwidth=1.75cm]{geometry}
\usepackage{graphicx}

\usepackage{xcolor}

\newtheorem{theorem}{Theorem}[section]
\newtheorem{corollary}[theorem]{Corollary}
\newtheorem{lemma}[theorem]{Lemma}
\newtheorem{prop}[theorem]{Proposition}

\theoremstyle{definition}
\newtheorem{definition}[theorem]{Definition}
\newtheorem{remark}[theorem]{Remark}

 \newtheorem*{thmA}{Theorem A} 
 \newtheorem*{thmB}{Theorem B} 
 \newtheorem*{thmC}{Theorem C} 
\newtheorem*{thmD}{Theorem D}

\newcommand\blfootnote[1]{%
  \begingroup
  \renewcommand\thefootnote{}\footnote{#1}%
  \addtocounter{footnote}{-1}%
  \endgroup
}
\newcommand{\PSH}{{\rm PSH}}

\newcommand{\capa}{{\rm Cap}}

\numberwithin{equation}{section}

 \usepackage{hyperref}
\hypersetup{  
    pdftitle={Cegrell classes},    
    pdfauthor={Mohammed Salouf},     
    colorlinks=true,       
   linkcolor=black,          
    citecolor=black,        
    filecolor=black,      
    urlcolor=black}        

\subjclass[2010]{32W20, 32U05, 32Q15, 35A23}

\keywords{Monge-Amp\`ere equations, hermitian manifolds, solutions with prescribed singularities, convergence in capacity}

\begin{document}

 \title[Degenerate complex Monge-Amp\`ere type equations]{Monge-Amp\`ere  equations with prescribed singularities on compact Hermitian manifolds} 
\author{Omar Alehyane}
\author{Chinh H. Lu} 
\author{Mohammed Salouf}

\address{
Laboratoire d'Informatique, Mathématiques et leurs Applications (LIMA), 
Faculté des Sciences El Jadida, 
Université Chouaib Doukkali,
24000 El Jadida, Maroc.} 
\email{alehyane.o@ucd.ac.ma}
\email{salouf.m@ucd.ac.ma}
\address{Institut Universitaire de France, Univ Angers, CNRS, LAREMA, SFR MATHSTIC, F-49000 Angers, France}
\email{hoangchinh.lu@univ-angers.fr}

\date{\today}
   
\maketitle

\begin{abstract}
 Given a compact complex manifold $X$, we study the existence and the uniqueness of weak solutions to degenerate Monge-Ampère equations on $X$ with prescribed singularities when the reference form is semipositive and big, while the right hand side is a non-pluripolar positive Radon measure. 
 This generalizes our previous work to more general hermitian manifolds and also to the case of solutions with prescribed singularities.
\end{abstract}

\blfootnote{This work is partially supported by the Kris project Fondation Charles Defforey and Institut Universitaire de France}

\tableofcontents
\section{Introduction}
The problem of solving Monge-Ampère equations {\rm (MAE)} on complex manifolds  has gained an increasing interest over the last few years due to numerous applications in differential geometry, algebraic geometry and theoretical physics.

In K\"ahler settings, first solutions to these equations  were established by Yau \cite{Yau78} and Aubin \cite{Aub78} with the aim of constructing canonical metrics as described in \cite{Cal57}. Bedford-Taylor \cite{BT76} showed that {\rm(MAE)} can be defined on all bounded quasi-plurisubharmonic functions, motivating the study of the existence of weak solutions. By adapting a construction that goes back to Bedford-Taylor \cite{BT87},  Guedj-Zeriahi \cite{GZ07} succeeded in  defining the Monge-Ampère operator on a class of unbounded quasi-plurisubharmonic functions and then established unbounded weak solutions to {\rm (MAE)} with a general right hand side being a non-pluripolar Radon measure. By this time, {\rm (MAE)} have continued to make a significant contribution in the development of K\"ahler geometry in connection with the search of canonical metrics and models (e.g. \cite{EGZ09,BCHM,DP10,BBGZ13,BG14,BB17,BBEGZ19,DR17,BDL17,BDL20,ChCh21a,ChCh21b}). In particular, by solving degenerate {\rm (MAE)}, Darvas,  Di Nezza and the second author \cite{DDL2,DDL4,DDL23} succeeded in constructing K\"ahler-Einstein metrics with prescribed singularities on singular K\"ahler varieties (we refer the reader to \cite{DDL23} and the references therein for a comprehensive survey of this topic and for numerous applications).

Studying {\rm (MAE)} on non-K\"ahler manifolds has always been a challenging issue.  Cherrier \cite{Cher87} has initiated the problem of  extending the results of Aubin and Yau  to non-K\"ahler settings. He succeeded only in a particular case by assuming a restrictive curvature condition.  Similar partial results have been given in \cite{Han96,GL10} before that Tossati-Weinkove solve the Monge-Ampère equation in full generality \cite{TW10}. Thereafter, many authors were interested in generalizing Tossati-Weinkove's result to the study of bounded solutions \cite{DK12, KN15Phong, Ngu16, GL22, GL23Crelle, KN23CVPDE,BGL24} when the right hand side is a measure having an $L^p$-density with respect to the Lebesgue measure.

Our first main result establishes bounded weak solutions to {\rm (MAE)} for measures dominated by the Monge-Ampère measure of bounded potentials. 
\begin{thmA}[Theorem \ref{thm: ext of bounded sol}]
Let $X$ be a compact complex manifolds of dimension $n$, and let $\theta$ be a smooth  $(1,1)$-form on $X$ which is semipositive and big. Assume $\mu$ is a positive Radon measure on $X$ satisfying $\mu \leq (\theta+dd^c \psi)^n$ with $\psi \in \PSH(X,\theta)\cap L^\infty(X)$. Then, for every positive constant $\lambda$, we can find a unique function $u \in  \PSH(X,\theta)\cap L^\infty(X)$ solving the equation
    $$ (\theta + dd^c u)^n = e^{\lambda u} \mu.  $$
\end{thmA}
When $X$ is K\"ahler and $\theta$ is closed, it is known by \cite[Corollary 11.9]{GZbook} that the equation in the last theorem admits a unique solution in the Guedj-Zeriahi class which is then bounded by the domination principle \cite[Proposition 10.11]{GZbook} (see also Corollary \ref{cor: uniqness MA-sol in [phi]}). Generalizing this result to the case $\lambda = 0$ will provide a positive answer to a question of Ko\l odziej from K\"ahler settings \cite[Question 15]{DGZ16}. The novelty here is the existence of the solution in a general complex manifold.

The main goal of this paper is to study the existence of unbounded solutions to {\rm (MAE)} on complex non-K\"ahler manifolds. In a previous work \cite{ALS24}, we considered the case of a complex compact manifold $X$ which satisfies a curvature condition introduced by Guan-Li. Then, giving any real smooth $(1,1)$-form $\theta$ on $X$ which is semipositive and big, we generalized the definition of the full mass class of Guedj-Zeriahi to this settings, we proved that the non-pluripolar Monge-Ampère measure is well defined on this class, that it is continuous with respect to convergence in capacity, and we solved the associated Monge-Ampère equation.

Our objective here is  to push forward the techniques developed in \cite{ALS24} to a more general complex manifold where we also seek for solutions with prescribed singularities in the terminology of \cite{DDL2}. We then let $(X,\omega)$ denote a compact hermitian manifold of dimension $n$, and we fix a smooth $(1,1)$-form $\theta$ on $X$ which is semipositive and big. As it is well known, a major difficulty of the problem lies in the fact that $\int_X (\theta+dd^c \varphi)^n$ depends on the chosen function $\varphi$. Guedj-Lu \cite{GL22} defined the quantities 
$$ v_+(\theta) = \sup \left\{ \int_X (\theta + dd^c \varphi)^n : \; \varphi \in \PSH(X,\theta) \cap L^\infty(X) \right\}, $$
and 
$$ v_-(\theta) = \inf \left\{ \int_X (\theta + dd^c \varphi)^n : \; \varphi \in \PSH(X,\theta) \cap L^\infty(X)  \right\}. $$
 It is an interesting open problem to see whether $v_-(\theta)>0$ and/or $v_+(\theta)<+\infty$. By Stokes' theorem, these conditions are always satisfied when $\theta$ is K\"ahler, or more generally Guan-Li. Corollary 2.4 in \cite{ALS24} shows that it holds true as soon as the manifold admits a hermitian metric $\omega$ satisfying  $v_+(\omega)<+\infty$ and $v_-(\omega)>0$, which covers the case of compact complex manifolds of dimension $n\leq 2$. At the time of this writing, we are not aware of any example of a compact complex manifold where any of these two conditions does not hold. We refer to \cite{GL22,AGL23,BGL24} for more information.

Let $\phi$ be a fixed $\theta$-model potential (see Definition \ref{defn: Mod pot}). Our second main result establishes weak solutions to degenerate Monge-Ampère equations with the same singularity type as $\phi$ in a general complex manifold, generalizing \cite[Theorem A]{DDL4} which deals with the case $\theta$ is closed and $\omega$ is K\"ahler.
\begin{thmB}[Theorems \ref{thm: sol to MA eqt with right hand-side has Lp density} and \ref{sol to MA for lambda=0}]
     Assume $\mu = f\omega^n$ where $0 \leq f \in L^p(\omega^n)$ is such that $p>1$ and $\mu(X)>0$. Then, 
      \begin{itemize}
        \item for every $\lambda>0$, there is a uniquely determined $\varphi \in \PSH(X,\theta)$ such that $\varphi \simeq \phi$ and
        $$ (\theta + dd^c \varphi)^n = e^{\lambda \varphi} \mu; $$
        \item  there is a unique constant $c>0$ and a function $\varphi \in \PSH(X,\theta)$ such that $\varphi \simeq \phi$ and
        $$ (\theta + dd^c \varphi)^n = c\mu. $$
    \end{itemize}
\end{thmB}
Here the notation $\varphi \simeq \phi$ means $\phi -C \leq \varphi \leq \phi + C$ for some constant $C\geq 0$; we say that $\varphi$ has the same singularity type as $\phi$. The function $\phi$ is not bounded (unless it vanishes identically), hence so is $\varphi$, the term $(\theta + dd^c \varphi)^n$ refers to the non-pluripolar Monge-Ampère measure defined in \cite{GZ07}.

To prove the last theorem, in the case $\theta$ is closed, Darvas-DiNezza-Lu \cite{DDL4} proposed a pluripotential method which is based on a generalization of Ko\l odziej's $L^\infty$ estimate \cite{Kol98}. This method involves minimizing the Monge-Ampère volume of potentials that are less singular than $\phi$: the function that achieves this minimum has a Monge-Ampère volume equal to the one of $\phi$ and it is less singular than $\phi$, hence it is the solution we seek. This method fails in our context since the total mass of the non-pluripolar Monge-Ampère measures of potentials with the same singularity type as $\phi$ is not constant. We could attempt to prove this theorem using the variational method in the spirit of \cite{DDL2,Tru22,DDL23}, however this seems unclear due to the difficulty of establishing a primitive of the operator $(\theta + dd^c .)^n$.

To overcome these difficulties, we consider a Perron type envelope of subsolutions $\varphi_j$ for which $\varphi_j \simeq \phi$ and the non-pluripolar Monge-Ampère measure of $\varphi_j$ is dominated by $\mu$. Then using Theorem A, we prove a subsolution type theorem for potentials with the same singularity type as $\phi$ (Theorem \ref{thm: sol to MA eqt with right hand-side is the sum of MA mesures}). We also establish a general relative uniform estimate of $\varphi_j$ (Lemma \ref{lem: uniform estimate}). Based on these two results, we prove that the total mass of the non-pluripolar Monge-Ampère measure of $\varphi_j$ converges to $\mu(X)$ and that $(\varphi_j)$ has a subsequence converging in capacity to the solution in the last theorem. The uniqueness of the solutions in the last theorem will follow from the domination principle (Theorem \ref{thm: the dom prn in [phi]}), which is a generalization of \cite[Proposition 2.8]{GL22} to the case of unbounded functions. Notice that the uniqueness of the solution $\varphi$ in the second part of the theorem remains largely open (we refer to \cite{KN19} for a partial answer).

After proving the last theorem, a natural question to ask is whether there are solutions to degenerate {\rm (MAE)} with right-hand side being a positive non-pluripolar Radon measure. As evidenced in \cite{GZ07,DiwJFA09,DDL2}, solutions to these equations belong to a special classes of quasi-psh functions. We thus let $\mathcal{E}(X,\theta,\phi)$ denote a generalization of the relative mass class defined in \cite{DDL2} and further studied in \cite{DDL4,DV22,Tru22,DDL23}. 
 
The following result gives a characterization of the range of the Monge-Ampère operator on the relative full mass class. 
\begin{thmC}[Theorem \ref{main thm C} and Corollary \ref{cor: sol of MA equation}]
    Assume $v_-(\theta)>0$. Let $\mu$ be a positive Radon measure vanishing on pluripolar sets. Then,
    \begin{itemize}
        \item for every $\lambda>0$, there is a uniquely determined $\varphi \in \mathcal{E}(X,\theta,\phi)$ such that 
        $$ (\theta + dd^c \varphi)^n = e^{\lambda \varphi} \mu; $$
        \item  there is a unique constant $c>0$ and a function $\varphi \in \mathcal{E}(X,\theta,\phi)$ such that 
        $$ (\theta + dd^c \varphi)^n = c\mu. $$
    \end{itemize}
\end{thmC} 
To prove the last theorem, we follow Cegrell's ideas from the local setting \cite{Ceg98,Ceg04}. This method is based on approximating the measure $\mu$ by appropriate Monge-Ampère measures with potentials $\varphi_j \in \mathcal{E}(X,\theta,\phi)$. This is possible thanks to Theorem B.

By construction, the sequence $\varphi_j$ satisfies one of the following conditions: 
\begin{enumerate}
    \item $(\theta + dd^c{\varphi_j})^n \leq A e^{\lambda \varphi_j}(\omega+dd^c{\psi_j})^n$, where  $\lambda$, $A$ are positive constants and $\psi_j$ is a uniformly bounded sequence of  $\omega$-psh functions converging in capacity to a function $\psi \in \PSH(X,\omega)\cap L^\infty(X)$; 
    \item $(\theta + dd^c{\varphi_j})^n \leq \nu$ is dominated by a non-pluripolar Radon measure $\nu$. 
\end{enumerate}
Applying the Hartogs' lemma and extracting, one can assume $\varphi_j$ converges in $L^1(X)$ to a $\theta$-psh function $\varphi$. We then face the challenge of proving that $\varphi$ is the solution in the last theorem. This will be done by proving the continuity of the non-pluripolar Monge-Ampère measure along the sequence $\varphi_j$. In general, this is not always possible for any sequence converging in $L^1(X)$ as showed by Cegrell  \cite[Example 3.25]{GZbook}. Some related results have been given in \cite{Xin09,DH12,BBEGZ19,GLZ19,DV22} when $\theta$ is closed and $X$ is K\"ahler. In a previous work \cite{ALS24}, by assuming that $X$ admits a Guan-Li metric and that $\phi=0$, we were able to prove that the convergence $\varphi_j \rightarrow \varphi$ holds in capacity and that the non-pluripolar Monge-Ampère measure is continuous with respect to this type of convergence.

Here we propose the following convergence criteria. 
\begin{thmD}[Theorem \ref{main cvg thm}]
     Assume $\varphi_j \in \mathcal{E}(X,\theta,\phi)$ satisfies (1) or (2) and converges in $L^1(X)$ to a $\theta$-psh function $\varphi$.  If $P_\theta(\inf_j \varphi_j) \in \mathcal{E}(X,\theta,\phi)$ then $\varphi \in \mathcal{E}(X,\theta,\phi)$, $\varphi_j \rightarrow \varphi$ in capacity and  $(\theta + dd^c{\varphi_j})^n \rightarrow (\theta + dd^c{\varphi})^n$ weakly. 
\end{thmD}
We will show in Lemma \ref{lem: P_theta(inf u_j) in E} that the condition $P_\theta(\inf_j \varphi_j) \in \mathcal{E}(X,\theta,\phi)$ can be assured by $v_-(\theta)>0$.

The paper is structured as follows.  In Section \ref{sect: NP measures}, we define the non-pluripolar Monge-Ampère measure of quasi-plurisubharmonic functions, and we give a sufficient condition to ensure that it is continuous with respect to potentials converging in capacity; we also study the Monge-Ampère volume of $\theta$-plurisubharmonic functions. In Section \ref{sect rel mass class}, we define and study the relative full mass class and we establish the domination principle (Theorems \ref{thm: the dom prn in [phi]} and \ref{thm: the dom prn in E(X,theta,phi)}). Section \ref{sec: cvg in capacity} is devoted to the proof of Theorem D. Theorems A, B and C are proven in Section 
\ref{sec: MAE with sin}. 

\section{Non-pluripolar Monge-Ampère measures}\label{sect: NP measures}
Throughout the paper, we let $(X,\omega)$ denote a compact hermitian manifold of dimension $n$ and we fix a smooth $(1,1)$-form $\theta$ on $X$ which is semipositive and big. Up to rescaling, there is no loss of generality in assuming $\theta \leq \omega$. 
\subsection{Quasi-plurisubharmonic functions}
We say that a function $u : X \rightarrow \mathbb{R} \cup \{-\infty\}$ is  quasi-plurisubharmonic, quasi-psh for short, if it is locally given by the sum of a smooth and a plurisubharmonic function.  We denote by $\PSH(X,\theta)$  the set of quasi-psh functions $u\in L^1(X)$ satisfying the condition  $\theta+ dd^c u\geq 0$ in the sense of currents. Functions in $\PSH(X,\theta)$ are called $\theta$-psh functions. If $u \in \PSH(X,\theta)$ then $\theta_u^n:=(\theta+ dd^c u)^n$ will refer to the non-pluripolar Monge-Ampère measure of $u$. It is defined by the strong increasing limit 
$$ \theta_u^n = \lim_{t \rightarrow +\infty} \textit{1}_{\{u>-t\}}(\theta + dd^c \max(u,-t))^n,  $$
see \cite{GZ07,ALS24}. We recall that the measures $(\theta + dd^c \max(u,-t))^n$ are defined by the Bedford-Taylor's wedge product \cite{BT76}. Observe that the condition \newline $v_+(\theta)<+\infty$ guarantees that $\theta_u^n$ is a Radon measure on $X$. Here we do not make use of the hypothesis $v_+(\theta)<+\infty$ and therefore we can not confirm that $\int_X \theta_u^n <+\infty$. However, Lemma \ref{lem : control of max(u,-j)} shows that $e^u \theta_u^n$ is always a Radon measure on $X$.  
\subsubsection{Maximum principle and quasi-psh envelopes}
Combining Theorem 2.1 and Corollary 2.2 in \cite{ALS24}, we can state the following fundamental result. 
\begin{theorem}\label{thm: plurifine property and Demailly type inequality} 
    Let $u$, $v\in \PSH(X,\theta)$. We have  
    $$  \textit{1}_{\{u>v\}} (\theta + dd^cu)^n = \textit{1}_{\{u>v\}} (\theta + dd^c \max(u,v))^n.   $$
    If in addition $u \leq v$, then 
    $$ \textit{1}_{\{u = v\}} (\theta + dd^cu)^n \leq \textit{1}_{\{u = v\}} (\theta + dd^cv)^n. $$ 
\end{theorem}
 Given any measurable function $h$ on $X$ with real values, we recall that the $\theta$-psh envelope of $h$ is defined by
$$ P_{\theta}(h) = \left(\sup \{ \varphi \in \PSH(X,\theta) : \varphi \leq h\}\right)^*,  $$
with the convention $\sup\emptyset = -\infty$. When $h$ is given by $\min(u,v)$ for some $u$, $v \in \PSH(X,\theta)$, we will use the abbreviation $P_\theta(u,v) := P_\theta(\min(u,v))$.

The following result will be used multiple times in the sequel. 
\begin{theorem}[Theorem 2.3 in \cite{ALS24}]\label{thm: env}
    Assume $h$ is quasi-continuous on $X$ and $P_\theta(h) \in \PSH(X,\theta)$. Then $P_\theta(h) \leq h$ outside a pluripolar set and 
    $$ \int_{\{P_\theta(h)<h\}} (\theta + dd^c P_\theta(h))^n = 0. $$ 
\end{theorem}
In the case where $h$ is given by the minimum of two $\theta$-psh functions, we have the following observation. 
\begin{prop}\label{prop: the envelope of min of two fncs}
Let $u$, $v \in \PSH(X,\theta)$. If $P_\theta(u,v) \in \PSH(X,\theta)$, then $$ (\theta + dd^c P_\theta(u,v))^n \leq \textit{1}_{\{P_\theta(u,v) = u < v\}} \theta_u^n + \textit{1}_{\{P_\theta(u,v) =  v\}} \theta_v^n.  $$
\end{prop}
\begin{proof}
Set $\varphi = P_\theta(u,v)$. By the last theorem, we know that the measure $\theta_\varphi^n$ is carried by the set 
   $$ \{\varphi = \min(u,v)\} = \{\varphi = u<v \} \cup \{\varphi = v\}.  $$
  Since $\varphi \leq \min(u,v)$,  we get by Theorem \ref{thm: plurifine property and Demailly type inequality} that
  $$    \theta_\varphi^n \leq   \textit{1}_{\{\varphi = u < v\}} \theta_\varphi^n + \textit{1}_{\{\varphi = v\}} \theta_\varphi^n  \leq \textit{1}_{\{\varphi = u < v\}} \theta_u^n + \textit{1}_{\{\varphi =v\}} \theta_v^n. $$
\end{proof}
We will also need the following consequence of Theorems \ref{thm: plurifine property and Demailly type inequality} and \ref{thm: env}. 
\begin{prop}\label{prop : the MA measure of P(bu - (b-1)v}
 Fix $b \geq 1$. Assume $u$, $v \in \PSH(X,\theta)$ are such that $\varphi_b := P_\theta(bu - (b-1)v)$ is in $\PSH(X,\theta)$. Then, we have 
 \begin{itemize}
     \item[(i)] $b^{-1} \varphi_b + (1-b^{-1}) v \leq u$ on $X$; 
     \item[(ii)] $\theta_{\varphi_b}^n$ is carried by $D_b := \{b^{-1} \varphi_b + (1-b^{-1}) v = u\}$;
     \item[(iii)] $b^{-n}\theta_{\varphi_b}^n + \textit{1}_{D_b} (1-b^{-1})^n  \theta_v^n \leq \textit{1}_{D_b} \theta_u^n$.
 \end{itemize} 
\end{prop}
\begin{proof}
By Theorem \ref{thm: env}, we know that $\theta_{\varphi_b}^n$ is carried by $D_b$ and that $b^{-1} \varphi_b + (1-b^{-1}) v \leq u$ a.e. on $X$. Since these functions are quasi-psh, we obtain $b^{-1} \varphi_b + (1-b^{-1}) v \leq u$ everywhere on $X$. Thus, we only need to prove (iii). 

According to Theorem \ref{thm: plurifine property and Demailly type inequality}, we have
$$ \textit{1}_{D_b} (\theta + dd^c({b^{-1}\varphi_b + (1-b^{-1})v}))^n \leq \textit{1}_{D_b} (\theta + dd^c u)^n. $$
The result then follows by the Newton binomial formula and (ii).
\end{proof}
We finish this paragraph with the following well-known lemmas. 
\begin{lemma}\label{lem: Ptheta(u-v) is in psh(X,theta)}
Assume $u$ and $v$ are quasi-psh functions on $X$. Then $P_\theta(u-v)$ is bounded from above and $P_\theta(u) - \sup_X v \leq P_\theta(u-v)$. 
\end{lemma}
\begin{proof}
Observe that the inequality $P_\theta(u) - \sup_X v \leq P_\theta(u-v)$ is obvious.

Without loss of generality we can assume $\sup_X v=0$.
Using this we see that the set $D := \{v>-1\}$ is non-pluripolar because $v$ is quasi-psh. It then follows from  \cite[Theorem 9.17]{GZbook} and \cite{Vu19IJM} that the global extremal function of $D$ with respect to $\omega$ 
$$ V_{D,\omega} := \sup \{h \in \PSH(X,\omega)\; : \; h \leq 0 \; \text{on } D \} $$
is bounded from above by some positive constant $A$. For  $C = \sup_X u +1$, we have $u-v \leq C$  on $D$,  hence $P_\theta(u-v) \leq V_{D,\varepsilon\omega} + C \leq A+C$ is bounded from above. 
The proof is complete.  
\end{proof}
\begin{lemma}\label{lem 11.5 GZbook}
    Let $\mu$ be a positive Radon measure on $X$ that vanishes on pluripolar sets. Assume $(u_j) \subset \PSH(X,\theta)$ is a uniformly bounded sequence that converges in $L^1(X)$ to a  function $u \in \PSH(X,\theta)$. Then $u_j \rightarrow u$ in $L^1(\mu)$. 
\end{lemma}
\begin{proof}
Using \cite[Lemma 11.5]{GZbook}, the proof goes exactly as in the one of \cite[Lemma 2.3]{KN22}. 
\end{proof}
\subsubsection{Convergence in capacity}
Recall that the capacity of a Borel set $D\subset X$, with respect to $\omega$, is defined by
$$ \capa_\omega(D) := \sup \left\{ \int_D (\omega + dd^c h)^n : h \in \PSH(X,\omega), 0\leq h\leq 1 \right\}.  $$
A sequence of measurable functions $(h_j)$ is said to be convergent in capacity to a measurable function $h$ if
$$ \lim_{j\rightarrow +\infty} \capa_{\omega}(\{|h_j - h|> \delta \}) = 0, $$
for every $\delta>0$.

It is well known, by an example of Cegrell \cite[Example 3.25]{GZbook}, that the Monge-Ampère operator is not continuous with respect to uniformly bounded sequences converging in $L^1(X)$. Building on \cite[Theorem 2.6]{DDL23}, we give a sufficient condition to ensure the continuity of the non-pluripolar Monge-Ampère measure with respect to  potentials converging in capacity.
\begin{lemma}\label{lem: continuity of MA mesures}
     Let  $u_j$, $u \in \PSH(X,\theta)$ be such that $\sup_j\int_X \theta_{u_j}^n<+\infty$ and $u_j \rightarrow u$ in capacity. Let $h_j$ be a uniformly bounded sequence of positive  and quasi-continuous functions that converge to a quasi-continuous function $h$ in capacity.  If the function $f:= P_\theta(\inf u_j)$ belongs to $\PSH(X,\theta)$ and 
  $$  \int_{\{f \leq -c\}} \theta_{u_j}^n   \rightarrow 0, \; \; \text{as} \; c\rightarrow +\infty $$
        uniformly with respect to $j$,
      then $h_j\theta_{u_j}^n \rightarrow h\theta_u^n$ in the weak sense of measures. 
\end{lemma}
\begin{proof}
By  \cite[Theorem 2.6]{DDL23}, we have $\int_X \theta_u^n <+\infty$.
Again by \cite[Theorem 2.6]{DDL23}, it suffices to prove 
\begin{equation}\label{eq : control of MAV}
    \limsup_{j\rightarrow +\infty}  \int_{X} \theta_{u_j}^n \leq \int_X \theta_u^n.
\end{equation}
Fix $\varepsilon$, $c>0$ and consider 
    $$  f_\varepsilon^c = \frac{\max(f+c,0)}{\max(f+c,0)+\varepsilon}. $$
    The function $f_\varepsilon^c$ is positive, bounded and quasi-continuous. Moreover, Theorem \ref{thm: plurifine property and Demailly type inequality} implies  
    $$ f_\varepsilon^c \theta_{u_j}^n = f_\varepsilon^c \theta_{\max(u_j,-c)}^n \;  \; \text{and} \; \; f_\varepsilon^c \theta_{u}^n = f_\varepsilon^c \theta_{\max(u,-c)}^n.    $$
  It then follows from \cite[Theorem 4.26]{GZbook} that 
    $$  \int_X f_{\varepsilon}^c \theta_u^n = \int_X f_{\varepsilon}^c \theta_{\max(u,-c)}^n   = \lim_{j\rightarrow +\infty}  \int_X f_{\varepsilon}^c \theta_{\max(u_j,-c)}^n  = \lim_{j\rightarrow +\infty}  \int_X f_{\varepsilon}^c \theta_{u_j}^n. $$
    Therefore, 
    $$ \int_X \theta_u^n \geq \limsup_{j\rightarrow +\infty}  \int_{\{f > -c + 1\}} f_{\varepsilon}^c \theta_{u_j}^n. $$
Using the fact that the map $x \mapsto \frac{x}{x+\varepsilon}$ is increasing on $\mathbb{R}^+$, we get $f_{\varepsilon}^c \geq  \frac{1}{1 + \varepsilon}$ on $\{f >-c + 1\}$, hence 
$$ \int_X \theta_u^n \geq \frac{1}{1 + \varepsilon} \limsup_{j\rightarrow +\infty} \int_{\{f > -c + 1\}} \theta_{u_j}^n. $$
By the hypothesis, letting $c \rightarrow +\infty$ and then $\varepsilon \rightarrow 0$ yields \eqref{eq : control of MAV}, which proves the lemma. 
\end{proof}
\subsection{Bounds on Monge-Ampère volumes}
Let $u$, $v \in \PSH(X,\theta)$. We say that $u$ is less singular than $v$, and we write $v\preceq u$,  if there is $C>0$ such that $v \leq u+C$ (we also say that $v$ is more singular than $u$). 

We define an equivalence relation on $\PSH(X,\theta)$ by requiring that $u \simeq v$ iff $u \preceq v$ and $v \preceq u$. Fix $u \in \PSH(X,\theta)$. The equivalence class of $u$ is the set of $\theta$-psh functions having the same singularity type as $u$, i.e. 
$$ [u] = \{ v \in \PSH(X,\theta) : v \simeq u \}. $$
In the sequel, we will need to bound the Monge-Ampère volume of functions in $[u]$. We then consider the quantities
$$ v_{-,u}(\theta) = \inf \left\{ \int_X \theta_v^n : v  \in \PSH(X,\theta), v \simeq u \right\}, $$
and 
$$ v_{+,u}(\theta) = \sup \left\{ \int_X \theta_v^n : v  \in \PSH(X,\theta), v \simeq u \right\}. $$
Let us list some basic properties of these Monge-Ampère volumes. 
\begin{itemize}
    \item   We have $\int_X \theta_v^n \leq v_+(\theta)$ for every $v \in \PSH(X,\theta)$, hence $v_{+,u}(\theta) \leq v_+(\theta)$.
    \item If $u$ is bounded then $v_{-,u}(\theta)= v_-(\theta)$ and $v_{+,u}(\theta)= v_+(\theta)$.
    \item When $\theta$ is closed, we have  $v_{-,u}(\theta)=v_{+,u}(\theta)= \int_X \theta_u^n$ by \cite[Theorem 1.2]{WN19JAMS}, hence $v_{+,u}(\theta)$ is always finite  and  $v_{-,u}(\theta)>0$ provided that $\int_X \theta_u^n>0$.  
\end{itemize}
The following result corresponds to \cite[Proposition 3.7]{BGL24}.
\begin{prop}\label{prop: comparison of MA vol}
Let $\theta_1 \leq \theta_2$ be two smooth $(1,1)$-forms on $X$ that are semipositive and big. If $u \in \PSH(X,\theta_1)$ and $v \in \PSH(X,\theta_2)$ are such that $u \preceq v$, then $v_{-,u}(\theta_1) \leq v_{-,v}(\theta_2)$ and $v_{+,u}(\theta_1) \leq v_{+,v}(\theta_2)$.
\end{prop}
In the rest of this paragraph, we will be concerned only with the condition $v_{-,u}(\theta)>0$. We first prove the following.
\begin{prop}\label{prop : 2.9}
 Assume $u \in \PSH(X,\theta)$ is such that  $v_{+,u}(\theta)> v_{+}(\theta) - v_{-}(\theta)$. Then there is $\varepsilon \in (0,1)$  such that $P_{\varepsilon\theta}(u)\not\equiv -\infty$.
\end{prop}
\begin{proof}
By hypothesis, we can find $b>1$ and $v \in \PSH(X,\theta)$ such that $v \simeq u$ and $\int_X \theta_v^n > v_{+}(\theta) - v_{-}(\theta)/b^n$. We are going to prove that $P_\theta(bv)$ belongs to $\PSH(X,\theta)$. It will then follow that   $P_\theta(bu)\in \PSH(X,\theta)$ because $v\simeq u$, hence it  suffices to take $\varepsilon = b^{-1}$. 

    For every $j$, we set $v_j = \max(v,-j)$ and $\varphi_j = P_\theta(bv_j)$. Observe that $\varphi_j$ is a bounded $\theta$-psh function. According to Proposition \ref{prop : the MA measure of P(bu - (b-1)v}, we have
    $$ \theta_{\varphi_j}^n = \textit{1}_{\{\varphi_j = bv_j\}}  \theta_{\varphi_j}^n \leq b^n \textit{1}_{\{\varphi_j = bv_j\}} \theta_{v_j}^n.  $$
    Assume by contradiction that $\sup_X \varphi_j \rightarrow -\infty$. Fixing $k\in \mathbb{N}$, we have $X = \{\varphi_j  \leq -bk\}$ for $j$ large enough. Hence 
    $$ v_-(\theta) \leq \int_{\{\varphi_j  \leq -bk\}} \theta_{\varphi_j}^n \leq b^n \int_{\{v_j \leq -k\}} \theta_{v_j}^n \leq b^n v_+(\theta) - b^n\int_{\{v > -k\}} \theta_v^n.  $$
    Letting $k \rightarrow +\infty$ gives
    $$ v_-(\theta) \leq b^n v_+(\theta) - b^n\int_{X} \theta_v^n, $$
    a contradiction. It follows that $P_\theta(bv) = \lim  v_j \in \PSH(X,\theta)$, which finishes the proof. 
\end{proof}
When $\theta$ is closed, we have $v_{-}(\theta) = v_{+}(\theta)$, hence the conclusion of the last proposition holds as soon as  $\int_X \theta_u^n>0$. The following proposition studies further this conclusion. 
\begin{prop}\label{prop: controling mass}
    Let $u \in \PSH(X,\theta)$ be such that $P_{\varepsilon \theta}(u)\not\equiv -\infty$ for some $\varepsilon \in (0,1)$. Then, there is $v \in \PSH(X,(1-\varepsilon)\theta) \cap L^\infty(X)$ such that 
    $$ \int_X ((1-\varepsilon)\theta + dd^c v)^n \leq \int_X (\theta + dd^c u)^n. $$
    In particular $v_{-,u}(\theta) \geq (1-\varepsilon)^n v_-(\theta)$. 
\end{prop}
\begin{proof}
   We set $\tilde{u}= P_{\varepsilon \theta}(u)$ and $v= P_{(1-\varepsilon)\theta}(u-\tilde{u})$. Since $\tilde{u} \leq u$, we get  by Lemma  \ref{lem: Ptheta(u-v) is in psh(X,theta)} that $v \in \PSH(X,(1-\varepsilon)\theta)\cap L^\infty(X)$. By Theorem \ref{thm: env}, we know 
that $v+\tilde{u} \leq u$ everywhere on $X$ and that $((1-\varepsilon)\theta + dd^c v)^n$ is carried by $\{v+\tilde{u}=u\}$. It then follows from Theorem \ref{thm: plurifine property and Demailly type inequality} that 
$$ \int_X ((1-\varepsilon)\theta + dd^c v)^n  \leq   \int_{\{v+\tilde{u}=u\}} (\theta + dd^c (v+\tilde{u}))^n \leq \int_X (\theta+ dd^c u)^n. $$
\end{proof}
Assuming $v_-(\theta)>0$ and $v_{+}(\theta)<+\infty$, we give a sufficient and necessary condition to ensure the conclusion of Proposition \ref{prop : 2.9}. 
\begin{prop}\label{condition : v_- positive implies model}
Assume $v_-(\theta)>0$ and $v_{+}(\theta)<+\infty$. For every $u \in \PSH(X,\theta)$, the following conditions are equivalent. 
\begin{itemize}
    \item[(i)] $v_{-,u}(\theta)>0$.
    \item[(ii)] there is $\varepsilon \in (0,1)$ such that $P_{\varepsilon \theta}(u)\not\equiv -\infty$.
\end{itemize}
\end{prop}
\begin{proof}
We start by the proof of (i) $\Rightarrow$ (ii).
 By contradiction, we assume that $P_\theta(bu) = -\infty$ for every $b>1$.  Fix $b>1$. 
For $t>0$, we consider $u_t=\max(bu,u-t))$ and $\phi_t=P_{\theta}(u_t)$.   Fix $C>0$. Since $\phi_t\searrow -\infty$ by assumption, we have $X = \{\phi_t\leq -C\}$ for $t$ sufficiently large.  It thus follows from Proposition \ref{prop : the MA measure of P(bu - (b-1)v} that
$$ v_{-,u}(\theta) \leq \int_{\{\phi_t = u_t \leq -C\}} (\theta+dd^c \phi_t)^n  \leq \int_{\{u \leq -C/b\}} (b\theta+dd^c u_t)^n, $$
where we have used that $\{u_t \leq -C\} = \{u\leq -C/b\}$ for every $t>0$ large enough. Observe that $u_t=u+ (b-1) \max(u,-t/(b-1))$. On one hand, we have by Theorem \ref{thm: plurifine property and Demailly type inequality}
 $$ \int_{\{-t/(b-1)<u\leq -C/b\}} (b\theta+dd^c u_t)^n \leq  b^n \int_{\{u\leq -C/b\}}(\theta+dd^c u)^n. $$
 On the other hand, by applying the Newton expansion formula and using $v_{+}(\theta)<+\infty$, we obtain
$$ \int_{\{u\leq -t/(b-1)\}} (b\theta+dd^c u_t)^n \leq \int_{\{u\leq -t/(b-1)\}}(\theta+dd^c u)^n + O(b-1). $$
We thus infer 
$$  v_{-,u}(\theta) \leq b^n \int_{\{u\leq -C/b\}}(\theta+dd^c u)^n + \int_{\{u\leq -t/(b-1)\}}(\theta+dd^c u)^n + O(b-1). $$
Letting $t \rightarrow +\infty$, $C \rightarrow +\infty$  and then $b \rightarrow 1$, we get $v_{-,u}(\theta)=0$, a contradiction.  This proves $P_\theta(bu) \not\equiv -\infty$ for some $b>1$, hence it suffices to take $\varepsilon = b^{-1}$.

The converse follows from the last proposition since $v_-(\theta)>0$.
\end{proof}
\subsection{Model potentials}\label{subsec: Model pot}
Fix $\varphi \in \PSH(X,\theta)$. If $\phi \in \PSH(X,\theta)$, then the envelope of singularity types of $\phi$ with respect to $\varphi$ is defined by 
$$ P_\theta[\phi](\varphi) = \left( \lim_{t \rightarrow +\infty} P_\theta(\phi+t,\varphi) \right)^*, $$
see \cite{RWN14}. 
When $\varphi=0$, we simply set $P_\theta[\phi]:=P_\theta[\phi](0)$. It is easy to see that $P_\theta[\phi] \in \PSH(X,\theta)$ and $\phi \preceq P_\theta[\phi]$. In the case $P_\theta[\phi] \simeq \phi$, we say that $\phi$ has model singularities.

The following proposition studies the non-pluripolar Monge-Ampère measure of the envelopes of singularity types. 
\begin{prop}\label{prop: the MA measure of P theta[u]}
 Let $u \in \PSH(X,\theta)$. Assume either $\int_X \theta_u^n <+\infty$ or $P_\theta[u] = u$. Then
 $$ (\theta + dd^c{P_\theta[u]})^n \leq \textit{1}_{\{P_\theta[u] =0\}} \theta^n.$$
\end{prop}
\begin{proof}
Fix $j\geq 1$ and set $u_j = P_\theta(u+j,0)$. Observe that  $u - \sup_X u \leq u_j$ and $u_j \nearrow P_\theta[u]$ almost everywhere. Observe also that
\begin{equation}\label{eq:lem: the MA measure of P theta[u]}
    \theta_{u_j}^n \leq \textit{1}_{\{u \leq -j\}} \theta_u^n + \textit{1}_{\{u_j =0\}} \theta^n,
\end{equation}
   according to  Proposition \ref{prop: the envelope of min of two fncs}.
 Assume  $\int_X \theta_u^n <+\infty$. By \eqref{eq:lem: the MA measure of P theta[u]}, the measures $\theta_{u_j}^n$ are dominated by a positive Radon measure.  It then follows from Lemma \ref{lem: continuity of MA mesures} that $\theta_{u_j}^n \rightarrow \theta_{P_\theta[u]}^n$ weakly. Therefore, letting $j \rightarrow +\infty$ in \eqref{eq:lem: the MA measure of P theta[u]} yields
  $$ \theta_{P_\theta[u]}^n \leq \textit{1}_{\{P_\theta[u] =0\}} \theta^n, $$
 since $\{u_j = 0\} \subset \{P_\theta[u] = 0\}$ for every $j$. 
 Assume now $P_\theta[u]=u$. In this case $u_j = u$ for every $j$, hence 
 $$  \theta_{u}^n \leq \textit{1}_{\{u \leq -j\}} \theta_u^n + \textit{1}_{\{u =0\}} \theta^n, $$
 by \eqref{eq:lem: the MA measure of P theta[u]}. Multiplying the above by $\textit{1}_{\{u>-t\}}$ where $t>0$ is fixed and using Theorem \ref{thm: plurifine property and Demailly type inequality}, we get 
 $$  \textit{1}_{\{u>-t\}} \theta_{\max(u,-t)}^n \leq \textit{1}_{\{-t < u \leq -j\}} \theta_{\max(u,-t)}^n + \textit{1}_{\{u =0\}} \theta^n. $$
 Letting $j \rightarrow +\infty$ we find 
 $$ \textit{1}_{\{u>-t\}} \theta_{\max(u,-t)}^n  \leq \textit{1}_{\{u =0\}} \theta^n. $$
 This is true for all $t>0$. Letting $t \rightarrow +\infty$, the result follows from the definition of the non-pluripolar Monge-Ampère measure.
\end{proof}
We will use the following terminology in the sequel. 
\begin{definition}\label{defn: Mod pot} 
 We say that a function $\phi \in \PSH(X,\theta)$  is a $\theta$-model potential if $P_\theta[\phi] = \phi$ and $P_{\varepsilon\theta}(\phi) \not\equiv -\infty$ for some $\varepsilon \in (0,1)$.   
\end{definition}
We have the following.
\begin{itemize}
    \item The function $\phi \equiv 0$ is a $\theta$-model potential on $X$. 
    \item Proposition \ref{prop: the MA measure of P theta[u]} ensures that 
    $\int_X \theta_\phi^n<+\infty$ for every $\theta$-model potential $\phi$.
    \item By Proposition \ref{prop: controling mass}, if $\theta$ is closed then every potential $\phi \in \PSH(X,\theta)$ that satisfies $P_\theta[\phi] = \phi$ and $\int_X \theta_\phi^n >0$ is $\theta$-model. Thus, the last definition is a generalization of the one introduced by Darvas, DiNezza and the second author (see the discussion following Theorem 2.1 in \cite{DDL4}).
    \item By Proposition \ref{condition : v_- positive implies model}, when $v_+(\theta)<+\infty$, the condition $P_{\varepsilon\theta}(\phi) \not\equiv -\infty$ for some $\varepsilon \in (0,1)$ is weaker than $v_{-,\phi}(\theta) >0$.
    \item According to Proposition \ref{prop: controling mass}, if $\phi$ is a $\theta$-model potential then
    $$ \int_X (\theta + dd^c u)^n > 0, \; \; \forall u \in \PSH(X,\theta),  u \simeq \phi. $$
    Indeed, if $u \in \PSH(X,\theta)$ satisfies $u \simeq \phi$ and $\int_X (\theta + dd^c u)^n = 0$ then Proposition \ref{prop: controling mass} yields $v \in \PSH(X,\varepsilon\theta) \cap L^\infty(X)$ such that $(\varepsilon \theta + dd^c v)^n = 0$. Arguing as in the proof of Corollary 2.4 in \cite{ALS24}, we can construct $w \in \PSH(X,\alpha \omega) \cap L^\infty(X)$ ($\alpha >0$) such that $(\alpha \omega + dd^c w)^n=0$, which is impossible.
\end{itemize}

We will give in Proposition \ref{prop: cons of model pot} a method for constructing $\theta$-model potentials, proving that there are plenty of these functions.
\section{The relative full mass class}\label{sect rel mass class}
\subsection{Definition and properties}
Fix $\phi \in \PSH(X,\theta)$. When $\theta$ is closed, the relative full mass class $\mathcal{E}(X,\theta,\phi)$  is defined by the set of $\theta$-psh functions $u$ satisfying $u\preceq \phi$ and 
$$ \int_X (\theta + dd^cu)^n = \int_X (\theta+dd^c\phi)^n, $$
see \cite{DDL2}.

When $\theta$ is not closed, the equality $\int_X \theta_u^n = \int_X \theta_\phi^n$ fails even when $u \simeq \phi$, hence the need for an alternative definition of $\mathcal{E}(X,\theta,\phi)$. This is the objective of the following proposition. 
\begin{prop}\label{prop: nat def}
    Assume $\theta$ is closed and $\phi \in \PSH(X,\theta)$ is such that $\int_X \theta_\phi^n>0$. Then a $\theta$-psh function $u$ belongs to $\mathcal{E}(X,\theta,\phi)$ iff  it is more singular than $\phi$ and satisfies
    $$ P_\theta(bu-(b-1)\phi) \in \PSH(X,\theta), \; \; \forall b \geq 1. $$
\end{prop}
\begin{proof}
    The direct implication follows from \cite[Theorem 3.10]{DDL23}. Hence, it suffices to prove the converse.
    
    Fix $b\geq 1$ and set $\varphi = P_\theta(bu-(b-1)\phi)$. The function  $\varphi$ belongs to $\PSH(X,\theta)$ by hypothesis, and we have $b^{-1}\varphi + (1-b^{-1})\phi \leq u$ by Proposition \ref{prop : the MA measure of P(bu - (b-1)v}. It then follows from the monotonicity of the non-pluripolar Monge-Ampère measure (see \cite{WN19}) that 
$$ (1-b^{-1})^n \int_X  \theta_\phi^n \leq  \int_X (\theta + dd^c(b^{-1}\varphi + (1-b^{-1})\phi))^n \leq \int_X \theta_u^n \leq \int_X  \theta_\phi^n . $$
Since $b$ is arbitrary, we obtain by letting $b \rightarrow +\infty$ that $\int_X \theta_u^n = \int_X \theta_\phi^n$, which implies $u \in \mathcal{E}(X,\theta,\phi)$. 
\end{proof}
We can thus extend the definition of the relative full mass class as follows.  
\begin{definition}
Let $\phi \in \PSH(X,\theta)$. We define the relative full mass class $\mathcal{E}(X,\theta,\phi)$ by the set of $\theta$-psh functions $u$ which are more singular than $\phi$ and satisfy the condition 
    $$ P_\theta(bu-(b-1)\phi) \in \PSH(X,\theta), \; \; \forall b \geq 1. $$
\end{definition}
 By definition, it is easy to see that the set $\mathcal{E}(X,\theta,\phi)$ is convex, that functions having the same singularity type as $\phi$ belong to $\mathcal{E}(X,\theta,\phi)$ and that 
    $$ u \in \mathcal{E}(X,\theta,\phi),  v \in \PSH(X,\theta), u \leq v \; \text{and } v \preceq \phi  \Longrightarrow  v \in \mathcal{E}(X,\theta,\phi). $$
    In particular, the relative full mass class is stable under maxima.

Observe also that, for $\phi=0$, the full mass class $\mathcal{E}(X,\theta):=\mathcal{E}(X,\theta,0)$ is then given by 
$$ \mathcal{E}(X,\theta) = \{ u \in \PSH(X,\theta) : P_\theta(bu) \in \PSH(X,\theta), \; \forall b \geq 1\}.  $$
According to \cite[Proposition 2.5]{ALS24}, when $X$ admits a Guan-Li metric, a $\theta$-psh function $u$ belongs to $\mathcal{E}(X,\theta)$ iff
$$  \lim_{t \rightarrow +\infty} \int_{\{u\leq -t \}} (\theta + dd^c \max(u,-t))^n = 0. $$
Therefore, this definition generalizes the one of Guedj-Zeriahi which was stated in the context of K\"ahler manifolds \cite{GZ07}.

The following proposition studies further the definition of the relative full mass class $\mathcal{E}(X,\theta,\phi)$ in the case $v_{-,\phi}(\theta)>0$.
\begin{prop}
Assume $\phi \in \PSH(X,\theta)$ satisfies $v_{-,\phi}(\theta)>0$, and let $u \in \PSH(X,\theta)$ be such that $u \preceq \phi$. If $\theta_u^n$ is a Radon measure on $X$ and
$$  \lim_{j \rightarrow +\infty} \int_{\{u \leq \phi - j\}} (\theta + dd^c \max(u,\phi-j))^n = 0 $$
then $u \in \mathcal{E}(X,\theta,\phi)$. The converse holds true if $X$ admits a Guan-Li form.
\end{prop}
It will be interesting to establish this result in a more general complex manifold.
\begin{proof}
We start by proving the direct implication. Fix $b \geq 1$. For every $j\geq 1$, we set $u_j = \max(u,\phi-j)$ and $\varphi_j = P_\theta(bu_j - (b-1) \phi)$. We have $\varphi_j \in \PSH(X,\theta)$, $\varphi_j \simeq \phi$ and $\varphi_j \searrow P_\theta(b u -(b-1) \phi)$. By contradiction, we assume $P_\theta(b u -(b-1) \phi) = -\infty$. It then follows by the Hartogs' lemma that $\sup_X \varphi_j \rightarrow -\infty$. Fixing $C\geq 0$, we can find $j_0 \geq 1$ such that $X = \{\varphi_j \leq -bC\}$ for every $j \geq j_0$.

Observe that  $\varphi_j \simeq \phi$ for every $j\geq 1$. Hence $\int_X \theta_{\varphi_j}^n \geq v_{-,\phi}(\theta)$. On the other hand, according to Proposition \ref{prop : the MA measure of P(bu - (b-1)v}, we know that 
\begin{align*}
    \int_X \theta_{\varphi_j}^n &\leq b^n \int_{\{\varphi_j = bu_j - (b-1) \phi \leq -bC\}} \theta_{u_j}^n \\
    &\leq b^n \int_{\{u \leq \phi - j\}} \theta_{u_j}^n + b^n \int_{\{u > \phi - j\} \cap \{u \leq -C\}} \theta_{u_j}^n.
\end{align*}
Therefore, letting $j\rightarrow +\infty$ we get by Theorem \ref{thm: plurifine property and Demailly type inequality} and the hypothesis that 
$$ v_{-,\phi}(\theta)   \leq b^n  \int_{\{u\leq -C\}} \theta_u^n. $$ 
 Now letting $C\rightarrow +\infty$, we find that $v_{-,\phi}(\theta)=0$, which is assumed not to be true by the hypothesis. We thus infer $P_\theta(bu - (b-1)\phi) \in \PSH(X,\theta)$. This holds for every $b\geq 1$, hence $u \in \mathcal{E}(X,\theta,\phi)$.

Assume now $\omega$ is Guan-Li and let $u \in \mathcal{E}(X,\theta,\phi)$. Since $\theta \leq \omega$, this implies $u \in \mathcal{E}(X,\omega,\phi)$, hence we get by  Proposition  \ref{prop: nat def} that 
$$ \int_X \omega_u^n = \int_X \omega_\phi^n = \int_X \omega_{\max(u,\phi -j)}^n, \; \; \forall j\geq 1. $$
It then follows from Theorem \ref{thm: plurifine property and Demailly type inequality} that 
$$ \int_X \omega_u^n = \int_{\{u>\phi -j\}} \omega_u^n + \int_{\{u \leq \phi -j\}} \omega_{\max(u,\phi -j)}^n, \; \; \forall j \geq 1. $$
Letting $j \rightarrow +\infty$ we obtain
$$ \lim_{j\rightarrow +\infty} \int_{\{u \leq \phi -j\}} (\omega + dd^c \max(u,\phi -j))^n = 0, $$
and the fact that  $\theta \leq \omega$ allows us to conclude. 
\end{proof}
We now give a sufficient condition to ensure that a $\theta$-psh function belongs to a relative full mass class. 
\begin{prop}\label{u in E(X,theta,Ptheta u)}
 If $u \in \PSH(X,\theta)$ is such that $v_{-,u}(\theta)>0$ and $\int_X \theta_u^n<+\infty$, then $u \in \mathcal{E}(X,\theta,P_\theta[u])$. 
\end{prop}
\begin{proof}
We have $u - \sup_X u \leq P_\theta[u]$. By contradiction, we assume there is
  $b \geq 1$ such that $P_\theta(b u - (b-1) P_\theta[u]) = -\infty$.  For every $j \geq 1$, we set $u_j = P_\theta(u+j,0)$ and $\varphi_j = P_\theta(b u - (b-1) u_j)$. The function $\varphi_j$ belongs to $\PSH(X,\theta)$ and satisfies $\varphi_j \simeq u$, in particular  $\int_X \theta_{\varphi_j}^n \geq v_{-,u}(\theta)$.

  Fix $k \in  \mathbb{N}$.  Since $\varphi_j \searrow -\infty$, we can find $j_0 \geq 1$ such that $X= \{\varphi_j \leq -bk\}$ for every $j \geq j_0$. By Proposition \ref{prop : the MA measure of P(bu - (b-1)v}, we know that 
    $$ \theta_{\varphi_j}^n \leq \textit{1}_{D_j} b^n \theta_u^n, $$
where $D_j = \{\varphi_j = bu - (b-1)u_j\}$.  Therefore
$$ v_{-,u}(\theta) \leq \int_X \theta_{\varphi_j}^n \leq  b^n\int_{\{u \leq -k\}}  \theta_u^n.  $$
 Since $k$ is arbitrary, we obtain by letting $k \rightarrow +\infty$ that $v_{-,u}(\theta)=0$, a contradiction. This proves $P_\theta(b u - (b-1) P_\theta[u]) \in \PSH(X,\theta)$ for every $b \geq 1$, hence $u \in \mathcal{E}(X,\theta,P_\theta[u])$. 
\end{proof}

\begin{prop}\label{prop: Ptheta(bu - (b-1)v) is in E}
    Let $\phi  \in \PSH(X,\theta)$ and let $u$, $v \in \mathcal{E}(X,\theta,\phi)$. We have $P_\theta(bu - (b-1)v) \in \mathcal{E}(X,\theta,\phi)$ for every $b\geq 1$. 
\end{prop}
\begin{proof}
     Subtracting a constant, one can assume $u$, $v \leq \phi$. Then, by the definition of $\mathcal{E}(X,\theta,\phi)$, it suffices to prove the result for $v=\phi$. Fix $b\geq 1$ and set $u_t = P_\theta(tu - (t-1)\phi)$ for every $t \geq 1$. Fixing $s\geq 1$ and $t>bs$, we have 
    $$ bu - (b-1)\phi = t^{-1}b(tu - (t-1)\phi) + (1-t^{-1}b) \phi,  $$
     hence $u_b \geq  t^{-1}b u_t+ (1-t^{-1}b) \phi$, and thus 
    $su_b - (s-1)\phi \geq t^{-1}bs u_t+ (1-t^{-1}bs) \phi$. 
   We then infer $P_\theta(su_b - (s-1)\phi) \in \PSH(X,\theta)$. Since $s$ is arbitrary, we get $u_b \in \mathcal{E}(X,\theta,\phi)$. 
\end{proof}
\begin{prop}\label{prop: control of MA mass in E}
Fix $\phi  \in \PSH(X,\theta)$. We have 
$$ P_\theta(u-\phi) \in \mathcal{E}(X,\theta), \; \; \forall u \in \mathcal{E}(X,\theta,\phi). $$
\end{prop}
\begin{proof}
  Fix $b\geq 1$ and set $\varphi = P_\theta(u-\phi)$. The function $\varphi$ belongs to $\PSH(X,\theta)$ by Lemma \ref{lem: Ptheta(u-v) is in psh(X,theta)}.
    Since $bu - (b-1) \phi \leq b(u-\phi)$, we get $P_\theta(b(u-\phi)) \in \PSH(X,\theta)$. 
    Using this  and the fact that  $b^{-1} P_\theta(b(u-\phi)) \leq \varphi$, we obtain  $P_\theta(b\varphi) \in \PSH(X,\theta)$. This holds for every $b\geq 1$, hence $\varphi \in \mathcal{E}(X,\theta)$. 
\end{proof}
\begin{prop}\label{prop: Ptheta(u,v) in E(X,theta,phi)}
 Assume $\phi$ is a $\theta$-model potential on $X$. If $u$, $v \in \mathcal{E}(X,\theta,\phi)$ then $P_\theta(u,v) \in \mathcal{E}(X,\theta,\phi)$. 
\end{prop}
\begin{proof}
Assume first that $P_\theta(u,v) \in \PSH(X,\theta)$ for every $u$, $v\in \mathcal{E}(X,\theta,\phi)$. Fix $u$, $v\in \mathcal{E}(X,\theta,\phi)$ and set $u_b = P_\theta(bu - (b-1)\phi)$ and $v_b = P_\theta(bv - (b-1)\phi)$ for every $b\geq 1$. By Proposition \ref{prop: Ptheta(bu - (b-1)v) is in E}, we have  $u_b$, $v_b \in \mathcal{E}(X,\theta,\phi)$,  hence  $P_\theta(u_b,v_b) \in \PSH(X,\theta)$ by our assumption. Since 
$$ b^{-1} P_\theta(u_b,v_b) + (1-b^{-1}) \phi \leq P_\theta(u,v),  $$
we get $P_\theta(u_b,v_b) \leq bP_\theta(u,v) - (b-1)\phi$,  hence $P_\theta( bP_\theta(u,v) - (b-1)\phi)\in \PSH(X,\theta)$. This holds for every $b\geq 1$, therefore $P_\theta(u,v) \in \mathcal{E}(X,\theta,\phi)$.

It thus remains to show $P_\theta(u,v) \in \PSH(X,\theta)$ for every $u$, $v \in \mathcal{E}(X,\theta,\phi)$. Fix $u$, $v \in \mathcal{E}(X,\theta,\phi)$ and set $\tilde{u} = P_\theta(u-\phi)$, $\tilde{v} = P_\theta(v-\phi)$. Clearly, there is no loss of generality in assuming $\max(u,v) \leq \phi$. First, Proposition \ref{prop: control of MA mass in E} gives  $\tilde{u}$, $\tilde{v}\in \mathcal{E}(X,\theta)$.  
Observe that $\tilde{u} + \phi \leq u$, $\tilde{v} + \phi \leq v$,   $\tilde{u} + \tilde{v} \leq \min(\tilde{u},\tilde{v})$ and that $P_\theta(2b\tilde{u} )/2 +P_\theta(2b\tilde{v} )/2 \leq bP_\theta(\tilde{u},\tilde{v})$ for every $b\geq 1$. Therefore 
$w := P_\theta(\tilde{u},\tilde{v}) \in \mathcal{E}(X,\theta)$ and $w + \phi \leq \min(u,v)$.  Since $\phi$ is a $\theta$-model potential, we know that $P_{\varepsilon\theta}(\phi) \not\equiv -\infty$ for some $\varepsilon \in (0,1)$. For $b> 1/(1-\varepsilon)$, the function $\varphi := b^{-1} P_\theta(bw) + P_{\varepsilon \theta} (\phi)$ belongs to $\PSH(X,\theta)$ and satisfies $\varphi \leq \min(u,v)$,  hence $P_\theta(u,v) \in \PSH(X,\theta)$. The result follows. 
\end{proof}

 The following proposition allows us to control the lower bound of the Monge-Ampère volumes of functions in the relative full mass class when $v_{-}(\theta)>0$.
\begin{prop}\label{prop: cond ensure v- positive}
  Assume $\phi \in \PSH(X,\theta)$  is such that $P_{\varepsilon\theta}(\phi) \not\equiv -\infty$ for some $\varepsilon \in (0,1)$.  For every $u \in \mathcal{E}(X,\theta,\phi)$,  we have 
    $$ \int_X (\theta + dd^c u)^n \geq (1-\varepsilon)^n v_-(\theta). $$
\end{prop}
\begin{proof}
Fix $\delta > 0$ small enough so that $\varepsilon + \delta < 1$. By the last proposition, there is $w \in \mathcal{E}(X,\theta)$ such that $u \geq \phi + w$. By the definition of $\mathcal{E}(X,\theta)$, we have  $P_\theta(\delta^{-1} w) \in \PSH(X,\theta)$ and $u \geq P_{\varepsilon\theta}(\phi) + \delta P_\theta(\delta^{-1} w)$, hence $P_{(\varepsilon + \delta)\theta}(u) \not\equiv -\infty$. We thus get by Proposition \ref{prop: controling mass} that 
   $$ \int_X \theta_u^n \geq (1-\varepsilon - \delta)^n v_-(\theta). $$
 The conclusion follows by letting $\delta \rightarrow 0$.
\end{proof}
\begin{remark}\label{rmk positivity of total masses}
The proof of the last proposition shows also that if $\phi$ is $\theta$-model potential on $X$, then 
$$ \int_X (\theta + dd^c u)^n > 0, \; \; \forall u \in \mathcal{E}(X,\theta,\phi). $$
\end{remark}
As a consequence of the last proposition, we obtain a method for constructing $\theta$-model potentials. 
\begin{prop}\label{prop: cons of model pot}
 Let $u \in \PSH(X,\theta)$ and set $\phi = P_\theta[u]$. If $v_{-,\phi}(\theta)>0$ and $v_+(\theta)<+\infty$ then $\phi$ is a $\theta$-model potential. 
\end{prop} 
\begin{proof}
     According to Proposition \ref{condition : v_- positive implies model}, there is $\varepsilon \in (0,1)$ such that $P_{\varepsilon \theta}(\phi)\not\equiv -\infty$. It thus suffices to prove $P_\theta[\phi]=\phi$.

     For every $j\geq 1$, we will use the notations $u_j = P_\theta(u+j,0)$ and $\phi_j = P_\theta(\phi+j,0)$. Fix $b \geq 1$ and consider $\phi_b = P_\theta(b\phi - (b-1)P_\theta[\phi])$. By Proposition \ref{u in E(X,theta,Ptheta u)}, we have  $\phi \in \mathcal{E}(X,\theta,P_\theta[\phi])$.  Proposition \ref{prop: Ptheta(bu - (b-1)v) is in E} then gives $\phi_b \in \mathcal{E}(X,\theta,P_\theta[\phi])$. 
  Moreover, we have $\phi \leq P_\theta[\phi]$, hence $\phi_b \leq \phi$ and $P_\theta(t\phi_b - (t-1)\phi) \geq P_\theta(t\phi_b - (t-1)P_\theta[\phi])$ for all $t \geq 1$, which means  $\phi_b \in \mathcal{E}(X,\theta,\phi)$. Since $v_{-,\phi}(\theta) \leq v_-(\theta)$ by Proposition \ref{prop: comparison of MA vol}, it follows from the last proposition that $\int_X \theta^n_{\phi_b} >0$.

  On the other hand, Proposition \ref{prop : the MA measure of P(bu - (b-1)v} implies  $\theta_{\phi_b}^n \leq b^n \textit{1}_{\{\phi_b = b\phi - (b-1)P_\theta[\phi]\}} \theta_\phi^n$. Since  the measure $\theta_\phi^n$ is carried by $\{\phi=0\}$ by Proposition \ref{prop: the MA measure of P theta[u]}, we infer that $\theta_{\phi_b}^n$ is carried by
 $$ \{\phi_b = b \phi - (b-1)P_\theta[\phi]\} \cap \{\phi = 0\} \subset \{\phi_b=0\}. $$
 Since  $\int_X \theta_{\phi_b}^n>0$, we get  $\sup_X \phi_b = 0$ for every $b$, hence $\psi := \lim \phi_b$ belongs to $\PSH(X,\theta)$. Fix $a>0$. We have  $\psi \leq b\phi - (b-1)P_\theta[\phi]$ for every $b$, hence $\psi =-\infty$ on $\{\phi<P_\theta[\phi] - a\}$, and thus  $\phi \geq P_\theta[\phi] - a$. Now letting $a \rightarrow 0$ gives $\phi \geq P_\theta[\phi]$, whence equality. 
\end{proof}

    From now on, and unless otherwise stated, we will denote by $\phi$ a fixed $\theta$-model potential on $X$.
\subsection{The domination principle}
The domination principle is an effective tool in the theory of complex Monge-Ampère equations. For example proving the uniqueness of solutions, connecting the $L^1$-convergence and the convergence in capacity of potentials, and thereby establishing weak solutions to Monge-Ampère equations  are all consequences of (or use in an essential way)  the domination principle \cite{GL22,DDL23, GL23Crelle, ALS24}.

We first propose the following version which generalizes \cite[Proposition 2.8]{GL22}. 
\begin{theorem}\label{thm: the dom prn in [phi]}
    Let $u \in \PSH(X,\theta)$ be such that $P_{\delta\theta}(u) \not\equiv - \infty$
    for some $\delta \in (0,1)$. Assume $v \in \PSH(X,\theta)$  is such that $v \preceq u$ and
    $$\theta_u^n \leq c \theta_v^n \; \; \text{on} \; \{u<v\}, $$ 
    for some $c \in [0,1)$. Then $u \geq v$.  
\end{theorem}
The proof of the last theorem is based on the following technical lemma. 
\begin{lemma}
   Let $u$ and $v \in \PSH(X,\theta)$ be such that $v \preceq u$ and $\theta_u^n = 0$ on $\{u<v\}$. If  $\delta \in (0,1)$ is such that $P_{\delta\theta}(v) \not\equiv -\infty$ then $u \geq P_{\delta\theta}(v)$. 
\end{lemma}
\begin{proof}
Set $v_\delta = P_{\delta\theta}(v)$ and define $\varphi = P_{(1-\delta)\theta} (u-v_\delta)$. By hypothesis, there is $C \geq 0$ such that $v - C \leq u$.  That implies $v_\delta - C \leq u$, and consequently $\varphi \in \PSH(X,(1-\delta)\theta) \cap L^\infty(X)$ by Lemma \ref{lem: Ptheta(u-v) is in psh(X,theta)}. It follows from Theorems \ref{thm: plurifine property and Demailly type inequality} and \ref{thm: env} that
$$ ((1-\delta)\theta + dd^c \varphi)^n \leq \textit{1}_{\{\varphi + v_\delta = u\}} ((1-\delta)\theta + dd^c \varphi)^n \leq \theta_u^n, $$
hence $((1-\delta)\theta + dd^c \varphi)^n(\{u<v\})=0$ by hypothesis. That implies
$$ ((1-\delta)\theta + dd^c \varphi)^n(\{\varphi<0\})=((1-\delta)\theta + dd^c \varphi)^n(\{\varphi + v_\delta = u\} \cap \{\varphi<0\})=0 $$
because $\{\varphi + v_\delta = u\} \cap \{\varphi<0\} \subset \{u<v\} \cup \{v_\delta = -\infty\}$. Proposition 2.8 in \cite{GL22} then yields $\varphi \geq 0$, and consequently $u \geq v_\delta$. 
\end{proof}
We now proceed to the proof of Theorem \ref{thm: the dom prn in [phi]}.
\begin{proof}
For $w:= \max(u,v)$, we have $w \in \PSH(X,\theta)$, $w \simeq u$ and   
$$ \theta_u^n \leq c \theta_w^n \; \; \text{on} \; \{u<w\} $$
by Theorem \ref{thm: plurifine property and Demailly type inequality}. Therefore, up to replacing $v$ by $\max(u,v)$, there is no loss of generality in assuming $v- C \leq u \leq v$ for some $C>0$. Fix $b \geq 1$ and consider $\varphi_b = P_\theta(bu - (b-1)v)$. We have $\varphi_b \in \PSH(X,\theta)$ and $-bC + u \leq \varphi_b \leq u$. 
Set $D:= \{\varphi_b=bu - (b-1)v\}$. According to Theorem \ref{thm: env}, the measure $\theta_{\varphi_b}^n$ is carried by $D$ and $\varphi_b \leq b u - (b-1)v$. 
   It thus follows from Theorem \ref{thm: plurifine property and Demailly type inequality} that  
    $$ \textit{1}_{D} (\theta + dd^c (b^{-1} \varphi_b + (1 - b^{-1})v))^n \leq \textit{1}_D \theta_u^n. $$
    Hence, we get by the hypothesis that
    $$ \textit{1}_{D\cap \{u<v\}} b^{-n}  \theta_{\varphi_b}^n + \textit{1}_{D\cap \{u<v\}}(1 - b^{-1})^n\theta_v^n \leq \textit{1}_{D \cap \{u<v\}} c\theta_v^n. $$
    Taking $b$ large enough so that $(1 - b^{-1})^n > c$, we obtain that  $\theta_{\varphi_b}^n$  is carried by  $D \cap \{u=v\} \subset \{\varphi_b = u\}$, hence $\theta_{\varphi_b}^n(\{\varphi_b <u\})=0$. It then follows from the last lemma that $\varphi_b \geq P_{\delta \theta}(u)$. Fix $a>0$. On the set $\{u<v-a\}$, we have $\varphi_b \leq v - ab$ for every $b$, hence $P_{\delta \theta}(u) = -\infty$.
 Using that $P_{\delta \theta}(u) \in L^1(X)$, we obtain $u \geq v -a$. The result now follows by letting $a\rightarrow 0$.
\end{proof}
As a consequence of Theorem \ref{thm: the dom prn in [phi]}, we have 
\begin{corollary}\label{cor: uniqness MA-sol in [phi]}
Let  $u$, $v \in \PSH(X,\theta)$ be such that  $v \preceq u$ and $P_{\delta\theta}(u) \not\equiv - \infty$ for some $\delta \in (0,1)$. 
\begin{itemize}
    \item[(i)] If  $e^{-\lambda u}\theta_u^n\leq e^{-\lambda v}\theta_v^n$ then $u\geq v$.  
    \item[(ii)] If $\theta_u^n \leq c \theta_v^n$ for some $c \geq 0$ then $c \geq 1$.
\end{itemize}
\end{corollary}
\begin{proof}
    Using Theorems \ref{thm: the dom prn in [phi]}, the proof goes immediately as in the ones of Corollaries 2.12 and 2.13 in \cite{ALS24}. 
\end{proof}
Assuming $v_-(\theta)>0$, we can extend the domination principle to the class $\mathcal{E}(X,\theta,\phi)$, providing a generalization of \cite[Theorem 3.12]{DDL23} and \cite[Theorem 1.3]{ALS24}.
\begin{theorem}\label{thm: the dom prn in E(X,theta,phi)}
Assume $v_-(\theta)>0$. Fix $c \in [0,1)$, and let $u \in \mathcal{E}(X,\theta,\phi)$ and $v \in \PSH(X,\theta)$ be such that $v \preceq \phi$. Assume either $\int_X \theta_u^n < +\infty$ or $\int_X \theta_v^n < +\infty$. If 
$$\theta_u^n \leq c \theta_v^n \; \; \text{on} \; \{u<v\}, $$ 
then $u \geq v$.  
\end{theorem}
\begin{proof}
The proof is similar to the one of \cite[Theorem 2.11]{ALS24}. Notice that the function $w:= \max(u,v)$ belongs to $\mathcal{E}(X,\theta,\phi)$, and it satisfies   
$$ \theta_u^n \leq c \theta_w^n \; \; \text{on} \; \{u<w\} $$
by Theorem \ref{thm: plurifine property and Demailly type inequality}. Therefore, up to replacing $v$ by $\max(u,v)$, there is no loss of generality in assuming $v \in \mathcal{E}(X,\theta,\phi)$ and $u \leq v$. 

 For every $b \geq 1$, the function $u_b := P_\theta(bu - (b-1)v)$ belongs to $\mathcal{E}(X,\theta,\phi)$ by Proposition \ref{prop: Ptheta(bu - (b-1)v) is in E}.  Using that $\theta_u^n$ or $\theta_v^n$ is a Radon measure on $X$ and arguing as in the proof of \cite[Theorem 2.11]{ALS24}, one can construct an increasing function $h: \mathbb{R}^+ \rightarrow \mathbb{R}^+$ such that $h(+\infty)=+\infty$ and 
 $$ \sup_{b\geq 1} \int_X h(|u_b|) \theta_{u_b}^n <+\infty. $$
If $\sup_X u_b \rightarrow -\infty$ then 
 $$  \int_X \theta_{u_b}^n \leq \frac{C}{h(|\sup_X u_b|)} \rightarrow 0, \; \text{as} \; b \rightarrow +\infty,$$
 contradicting Proposition \ref{prop: cond ensure v- positive}. Therefore $(u_b)$ has a subsequence that converges in $L^1(X)$ to a function $\varphi \in \PSH(X,\theta)$.

Fix $a>0$. On the set $\{u<v-a\}$, we have $u_b \leq v - ab$ for every $b$, hence $\varphi = -\infty$. Since $\varphi \in L^1(X)$, we obtain $u \geq v -a$. This holds for every $a>0$, hence $u\geq v$. 
\end{proof}
As above, we have the following consequence of Theorem \ref{thm: the dom prn in E(X,theta,phi)}. 
\begin{corollary}\label{cor: uniqness MA-sol in E(X,theta,phi)}
Assume $v_-(\theta)>0$. Let  $u \in \mathcal{E}(X,\theta,\phi)$ and $v \in \PSH(X,\theta)$ be such that $v \preceq \phi$ and $\int_X \theta_u^n<+\infty$. 
\begin{itemize}
    \item[(i)] If  $e^{-\lambda u}\theta_u^n\leq e^{-\lambda v}\theta_v^n$ then $u\geq v$.  
    \item[(ii)] If $\theta_u^n \leq c \theta_v^n$ for some $c \geq 0$ then $c \geq 1$.
\end{itemize}
\end{corollary}
\section{Weak convergence of Monge-Ampère measures}\label{sec: cvg in capacity}
The objective of this section is to develop pluripotential tools that will be of great help in solving degenerate Monge-Ampère equations.  By following Cegrell's ideas from local settings \cite{Ceg98,Ceg04}, we are led to work with sequences $u_j \in \PSH(X,\theta)$ that satisfy one of the following conditions: 
\begin{enumerate}
    \item  $\theta_{u_j}^n \leq A e^{\lambda u_j}\omega_{\psi_j}^n$, where  $\lambda$, $A$ are positive constants and $\psi_j$ is a uniformly bounded sequence of  $\omega$-psh functions converging in capacity to a function $\psi \in \PSH(X,\omega)\cap L^\infty(X)$; \label{condition 1}
    \item  $\theta_{u_j}^n \leq \mu$ is dominated by a non-pluripolar Radon measure $\mu$. \label{condition 2}
\end{enumerate}
The main objective of this section is to prove the following. 
\begin{theorem}\label{main cvg thm}
    Let $u_j \in \mathcal{E}(X,\theta,\phi)$ be a sequence converging in $L^1(X)$ to a function $u \in \mathcal{E}(X,\theta,\phi)$ such that $P_\theta(\inf u_j) \in \mathcal{E}(X,\theta,\phi)$. Assume either $v_{-}(\theta)>0$ or $P_\theta(\inf u_j) \simeq \phi$.  If $u_j$ satisfies \eqref{condition 1} or \eqref{condition 2}, then $u_j \rightarrow u$ in capacity and $\theta_{u_j}^n \rightarrow \theta_u^n$ weakly.   
\end{theorem}
Observe that the sequence $u_j$ in the last theorem  satisfies $\sup_j \int_X \theta_{u_j}^n < +\infty$. 
Before proving the theorem, we will give a sufficient condition ensuring $P_\theta(\inf u_j) \in \mathcal{E}(X,\theta,\phi)$.
\begin{lemma}\label{lem: P_theta(inf u_j) in E}
      Assume $v_{-}(\theta)>0$. Let $u_j \in \mathcal{E}(X,\theta,\phi)$ be a sequence that satisfies \eqref{condition 1} or \eqref{condition 2}. If  $u_j \rightarrow u \in \PSH(X,\theta)$ in $L^1(X)$  then $u$ and $P_\theta(\inf u_j)$ are in $\mathcal{E}(X,\theta,\phi)$. 
\end{lemma}
\begin{proof}
  Observe that $u_j \preceq \phi$ for every $j$. Since $u_j \rightarrow u \in \PSH(X,\theta)$ in $L^1(X)$, the Hartogs' lemma implies $(\sup_X u_j)$ is bounded. Using this and the fact that $\phi$ is a $\theta$-model potential, we can find a uniform constant $C$ such that $u_j \leq \phi + C$. We thus infer $u \preceq \phi$.  
  
  By the definition of $\mathcal{E}(X,\theta,\phi)$, we will be done if we prove $P_\theta(\inf u_j) \in \mathcal{E}(X,\theta,\phi)$. Fix then $b \geq 1$ and set   $v_j = P_\theta(u_1,...,u_j)$, $v = P_\theta(\inf u_j)$ and $\varphi_j = P_\theta(bv_j - (b-1)\phi)$. We want to prove that $(\sup_X \varphi_j)$ doesn't converge to $-\infty$. Assume by contradiction that this is not true. Up to extracting, we can further assume $\varphi_j \leq -j$ for every $j$.

  According to Proposition \ref{prop: Ptheta(u,v) in E(X,theta,phi)}, we have $v_j \in \mathcal{E}(X,\theta,\phi)$, hence $\varphi_j \in \mathcal{E}(X,\theta,\phi)$ by Proposition \ref{prop: Ptheta(bu - (b-1)v) is in E}. 
  In particular, since $v_-(\theta)>0$, the sequence  $\int_X \theta_{\varphi_j}^n$ is uniformly bounded away from zero by Proposition \ref{prop: cond ensure v- positive}.

  By Theorem \ref{thm: env}, we know that $\theta_{\varphi_j}^n$ is carried by the set $D_j = \{\varphi_j = bv_j - (b-1)\phi\}$. Since the functions $b^{-1}\varphi_j + (1-b^{-1})\phi$ and $v_j$ belong to $\PSH(X,\theta)$ and satisfy $b^{-1}\varphi_j + (1-b^{-1})\phi \leq v_j$ with equality on $D_j$,  we get by Theorem \ref{thm: plurifine property and Demailly type inequality} that
\begin{equation}\label{dom of varphi_j}
    \theta_{\varphi_j}^n \leq b^n \textit{1}_{D_j} (\theta + dd^c(b^{-1}\varphi_j + (1-b^{-1})\phi))^n \leq b^n \textit{1}_{D_j} \theta_{v_j}^n.
\end{equation}
     Assume first that $u_j$ satisfies \eqref{condition 1}. Applying Proposition \ref{prop: the envelope of min of two fncs} yields
    \[     \int_X e^{-\lambda v_j} \theta_{v_j}^n \leq  A \sum_{k=1}^j \int_X \omega_{\psi_k}^n.   \]
    Since $(\psi_k)$ is uniformly bounded, we can find $C\geq 0$ such that 
    \[ \int_X e^{-\lambda v_j} \theta_{v_j}^n \leq C j. \]
  It follows that from \eqref{dom of varphi_j} that 
   \[\int_X e^{-\lambda b^{-1} \varphi_j}\theta_{\varphi_j}^n \leq \int_X e^{-\lambda (b^{-1} \varphi_j + (1-b^{-1})\phi) } \theta_{\varphi_j}^n
 \leq  b^n C j.  \]
 Therefore,  
   \[  \int_X \theta_{\varphi_j}^n \leq e^{\lambda b^{-1} \sup_X \varphi_j} \int_X e^{-\lambda b^{-1} \varphi_j}\theta_{\varphi_j}^n  \leq C je^{-\lambda b^{-1} j},
   \]
hence $\inf \int_X \theta_{\varphi_j}^n=0$, a contradiction. 
     Assume now that $u_j$ satisfies \eqref{condition 2}. By Lemma \ref{lem 11.5 GZbook}, we have
     $$ \lim_{j\rightarrow +\infty} \int_X |e^{u_{j}} - e^u| \theta_{u_{j}}^n = 0. $$
     Passing to a subsequence if necessary, we can assume 
     $$ \int_X |e^{u_{j}} - e^u| \theta_{u_{j}}^n \leq 2^{-j}, \; \; \forall j \geq 1. $$
   Proposition \ref{prop: the envelope of min of two fncs} then implies  
\[  \int_X |e^{v_{j}} - e^u| \theta_{v_{j}}^n \leq 2^{-j+1}.\]
Hence
$$ \int_X  |e^{b^{-1}\varphi_{j}+ (1-b^{-1})\phi} - e^u| \theta_{\varphi_{j}}^n  \leq b^n 2^{-j+1}, $$
by \eqref{dom of varphi_j}. Again by \eqref{dom of varphi_j},  we have
$$ \int_X e^{b^{-1}\varphi_{j}+ (1-b^{-1})\phi}\theta_{\varphi_j}^n \leq e^{-b^{-1}j} \int_X  \theta_{v_j}^n \leq j e^{-b^{-1}j} \mu(X), $$
where in the last inequality we have used Proposition \ref{prop: the envelope of min of two fncs}. Finally, we can write 
$$ \int_X e^u \theta_{\varphi_{j}}^n \leq b^n 2^{-j+1} + \int_X e^{b^{-1}\varphi_{j}+ (1-b^{-1})\phi}\theta_{\varphi_j}^n \leq b^n 2^{-j+1} + j e^{-b^{-1}j} \mu(X) \to 0 $$
as $j\rightarrow +\infty$. Fixing $C>0$, we have 
$$ \int_X \theta_{\varphi_{j}}^n \leq \mu(\{u \leq -C\}) +  e^C\int_X e^u \theta_{\varphi_{j}}^n. $$
Letting $j \rightarrow +\infty$ and then $C\rightarrow +\infty$, we find that $\inf \int_X \theta_{\varphi_{j}}^n = 0$, 
a contradiction.

We infer that in both cases  $(\sup_X \varphi_j)$ is bounded from below.  Since $\varphi_j$ decreases to  $P_\theta(bv - (b-1)\phi)$ and $b$ is arbitrary, we get $v \in \mathcal{E}(X,\theta,\phi)$, which proves the result. 
\end{proof}
The following lemma gives a sufficient condition to ensure the convergence in capacity of $\theta$-psh functions that converge in $L^1(X)$, generalizing \cite[Theorem 3.1]{ALS24} which deals with uniformly bounded sequences. 
\begin{lemma}\label{lem: cvg in capacity}
Let $u_j \in \PSH(X,\theta)$ be such that $P_\theta(\inf u_j) \in \mathcal{E}(X,\theta,\phi)$ and $\sup_j \int_X \theta_{u_j}^n<+\infty$. Assume either $P_\theta(\inf u_j) \simeq \phi$ or $v_-(\theta)>0$. If $u_j$ converges in $L^1(X)$ to a function $u \in \mathcal{E}(X,\theta,\phi)$  and  
    $$ \lim_{j\rightarrow +\infty} \int_X |e^{u_j} - e^{u}| \theta_{u_j}^n = 0, $$
   then $u_j \rightarrow u$ in capacity. 
\end{lemma}
\begin{proof}
The proof follows the lines of the proof of \cite[Theorem 3.1]{ALS24} which builds on the rooftop envelopes as in \cite{Dar15,Dar17}.

Up to extracting a subsequence, one can assume 
$$ \int_X |e^{u_j} - e^{u}| \theta_{u_j}^n \leq 2^{-j}, \; \; \forall j \geq 1. $$    
For any $j$, $k \in \mathbb{N}$, we set $u_{j,k}= P_\theta(u_j,...,u_{j+k})$. By hypothesis, the function $u_{j,k}$ belongs to $\PSH(X,\theta)$ and it converges to the $\theta$-psh function $\tilde{u}_j := P_\theta(\inf_{k\geq j} u_k)$. By Proposition \ref{prop: the envelope of min of two fncs}, we have 
 $$ \int_X |e^{u_{j,k}} - e^{u}| \theta_{u_{j,k}}^n  \leq 2^{-j+1}, \; \; \forall j, k \geq 1. $$ 
 Therefore, letting $k \rightarrow +\infty$, we get by \cite[Theorem 2.6]{DDL23} that 
 $$ \int_X |e^{\tilde{u}_j} - e^{u}| \theta_{\tilde{u}_j}^n \leq \liminf_{k\rightarrow +\infty} \int_X |e^{u_{j,k}} - e^{u}| \theta_{u_{j,k}}^n   \leq 2^{-j+1},  $$ 
 for all $j \geq 1$.  Observe now that $(\tilde{u}_j)$ is increasing, hence it converges a.e. to a function $\tilde{u} \in \PSH(X,\theta)$.  We get similarly that 
 \begin{equation*}
     \int_X |e^{\tilde{u}} - e^{u}| \theta_{\tilde{u}}^n \leq \lim_{j\rightarrow +\infty} \int_X |e^{\tilde{u}_j} - e^{u}| \theta_{\tilde{u}_j}^n = 0.
 \end{equation*}
 By hypothesis, we have  $u$, $P_\theta(\inf u_j) \in \mathcal{E}(X,\theta,\phi)$. Since $P_\theta(\inf u_j) \leq \tilde{u} \leq u$ by construction, it follows that $\tilde{u} \in \mathcal{E}(X,\theta,\phi)$. Observe that if $P_\theta(\inf u_j) \simeq \phi$ then so does $\tilde{u}$. Observe also that $\int_X \theta_u^n < +\infty$ by Theorem 2.6 in \cite{DDL23}. Therefore, we infer by  Theorems \ref{thm: the dom prn in [phi]} and \ref{thm: the dom prn in E(X,theta,phi)} that $u \leq \tilde{u}$, hence $u = \tilde{u}$.

 Finally, we have $\tilde{u}_j \leq u_j \leq \max(u,u_j)$ and the sequences $(\tilde{u}_j)$  and $(\max(u,u_j))$ converge to $u$ in capacity. Therefore $u_j \rightarrow u$ in capacity, completing the proof.
\end{proof}
We are now in position to prove Theorem \ref{main cvg thm}.  
\begin{proof}[Proof of Theorem \ref{main cvg thm}]
 By hypothesis, the sequence $(u_j)$ converges in $L^1(X)$ to $u$. Applying the Hartogs' lemma, we infer that $(\sup_X u_j)$ is bounded, hence $\theta_{u_j}^n \leq C \omega_{\psi_j}^n + \mu$ for some constant $C>0$. It thus follows from \cite[Lemma 2.3]{KN22} and Lemma \ref{lem 11.5 GZbook} that 
  $$ \lim_{j\rightarrow +\infty} \int_X |e^{u_j} - e^u| \theta_{u_j}^n =0. $$
 Lemma \ref{lem: cvg in capacity} then implies $u_j \rightarrow u$ in capacity. It thus remains to prove $\theta_{u_j}^n \rightarrow \theta_u^n$ weakly. Fix then $c>0$ and set $f = P_\theta(\inf u_j)$. Since $u_j$ satisfies the condition \eqref{condition 1} or \eqref{condition 2}, we have 
    $$ \int_{\{f \leq -c\}} \theta_{u_j}^n \leq C \capa_\omega(\{f \leq -c\}) + \mu(\{f \leq -c\}),  $$
    for some constant $C>0$.  It follows from \cite[Proposition 2.5]{DK12} that 
    $$  \int_{\{f \leq -c\}} \theta_{u_j}^n   \rightarrow 0, \; \; \text{as} \; c\rightarrow +\infty $$
        uniformly with respect to $j$, hence $\theta_{u_j}^n \rightarrow \theta_u^n$ weakly by Lemma \ref{lem: continuity of MA mesures}, finishing the proof. 
\end{proof}
\section{Monge-Ampère equations with prescribed singularities}\label{sec: MAE with sin}
This section is devoted to the main problem of solving degenerate Monge-Ampère equations in the class $\mathcal{E}(X,\theta,\phi)$. We will first start by studying the case when the right hand side has a $L^p$-density with respect to the Lebesgue measure, and then we will proceed to the general case.  
\subsection{The case of measures with $L^p$-density}
Our objective in this paragraph is to solve the Monge-Ampère equation 
$$ (\theta + dd^cu)^n = e^{\lambda u} f \omega^n, \; \; u \in \PSH(X,\theta) \; \text{with} \; u \simeq \phi. $$
Here $\lambda\geq 0$,  $p>1$  are constants and  $0 \leq f \in L^p(X)$ is such that $\int_X f \omega^n >0$. 
In the sequel, we will assume $\lambda>0$. The case $\lambda=0$ will be treated separately in Theorem \ref{sol to MA for lambda=0}.

We first start by proving the following subsolution theorem. 
\begin{theorem}\label{thm: ext of bounded sol}
Assume $\mu = f (\theta+dd^c \psi)^n$ where $\psi \in \PSH(X,\theta)\cap L^\infty(X)$ and $0 \leq f \in L^\infty(X)$ is such that $\mu(X)>0$. Then, we can find a unique function $u \in  \PSH(X,\theta)\cap L^\infty(X)$ such that 
    $$ (\theta + dd^c u)^n = e^{\lambda u} \mu.  $$
\end{theorem}
Observe that generalizing this result to the case $\lambda=0$ will provide a positive answer to a question of Ko\l odziej from K\"ahler settings \cite[Question 15]{DGZ16}. 
\begin{proof}  We divide the proof into two steps.

{\bf Step 1.} We first assume $f \in \mathcal{C}^\infty(X)$, $f>0$. Let $\psi_j$ be a sequence of smooth $\omega$-psh functions decreasing to $\psi$.  The sequence  $\varphi_j := P_\theta(\psi_j)$ is a uniformly bounded sequence of $\theta$-psh functions which decreases to $\psi$ since $\psi \in \PSH(X,\theta)$. Proposition \ref{prop : the MA measure of P(bu - (b-1)v}  then implies
$$ \theta_{\varphi_j}^n = \textit{1}_{\{\varphi_j = \psi_j\}} \theta_{\varphi_j}^n \leq \omega_{\psi_j}^n, $$
hence the measure $f\theta_{\varphi_j}^n$ has a $L^\infty$-density with respect to the Lebesgue measure. Applying  \cite[Theorem 3.4]{GL23Crelle} yields $u_j \in \PSH(X,\theta)\cap L^\infty(X)$ such that 
$$ (\theta + dd^c u_j)^n = e^{\lambda u_j} f(\theta + dd^c\varphi_j)^n. $$
Since $f>0$ is bounded and $\varphi_j$ is uniformly bounded, \cite[Corollary 2.9]{GL22} gives $C>0$ such that $\varphi_j - C \leq u_j \leq\varphi_j + C$, hence $u_j$ is uniformly bounded. Passing to a subsequence if necessary, one can assume $u_j \rightarrow u \in \PSH(X,\theta)\cap L^\infty(X)$ in $L^1(X)$. By definition of $u_j$, we have $\theta_{u_j}^n \leq\|f\|_\infty e^{\lambda u_j} \theta_{\varphi_j}^n$, hence 
$$ \lim_{j\rightarrow +\infty} \int_X |e^{u_j} - e^{u}| \theta_{u_j}^n = 0 $$
by \cite[Lemma 2.3]{KN22}. It follows from Lemma \ref{lem: cvg in capacity} that $u_j \rightarrow u$ in capacity. Therefore, we get by \cite[Theorem 4.26]{GZbook} that  $\theta_{u_j}^n \rightarrow \theta_u^n$ weakly and
$$ \theta_{u_j}^n = e^{\lambda u_j} f \theta_{\varphi_j}^n \rightarrow e^{\lambda u} f \theta_\psi^n \; \; \text{weakly,} $$
which means $\theta_u^n = e^{\lambda u} f \theta_\psi^n$.

    {\bf Step 2.}   We now proceed to the general case when $f$ is merely bounded. Let $f_j$ be a uniformly bounded sequence of positive smooth functions converging to $f$ in $L^1(\theta_\psi^n)$ and a.e. with respect to $\theta_\psi^n$. For each $j$, Step 1 gives $u_j \in \PSH(X,\theta)\cap L^\infty(X)$ such that 
      $$  \theta_{u_j}^n = e^{\lambda u_j} f_j \theta_{\psi}^n.  $$
      By construction, the function $u_j$ is the limit in capacity of a uniformly bounded sequence $(u_j^k) \subset \PSH(X,\theta)$ such that
      $$ (\theta + dd^c u_j^k)^n = e^{\lambda u_j^k} f_j \theta_{\psi_k}^n,  $$
      where $\psi_k \in \PSH(X,\theta) \cap L^\infty(X)$ is such that $\psi_k \searrow \psi$ and $\theta_{\psi_k}^n$ is absolutely continuous with respect to the Lebesgue measure.
      It follows from  the mixed Monge-Ampère inequality \cite[Lemma 1.9]{Ng16AIM}  that 
      $$ (\theta + dd^c u_j^k) \wedge (\theta + dd^c \psi_k)^{n-1}  \geq e^{\lambda u_j^k/n} f_j^{1/n} (\theta + dd^c \psi_k)^n.  $$
      Letting $k \rightarrow +\infty$, we get by \cite[Theorem 4.26]{GZbook} that 
      \begin{equation}\label{Mixed MA ineq}
          (\theta + dd^c u_j) \wedge (\theta + dd^c \psi)^{n-1}  \geq e^{\lambda u_j/n} f_j^{1/n} (\theta + dd^c \psi)^n. 
      \end{equation}
          Consider now $v_j = u_j - \sup_X u_j$ and $c_j = e^{\lambda \sup_X u_j}$. By the Hartogs lemma, we know that  a subsequence of $v_j$ converges in $L^1(X)$ to a function $v \in \PSH(X,\theta)$. Using Lemma \ref{lem 11.5 GZbook} and extracting if necessary, we can assume $v_j \rightarrow v$ a.e. with respect to $\theta_\psi^n$.  In particular, we get by the Lebesgue  convergence theorem that
          $$ \int_X e^{\lambda v_j/n} f_j^{1/n} (\theta + dd^c \psi)^n \rightarrow \int_X e^{\lambda v/n} f^{1/n} (\theta + dd^c \psi)^n.  $$
          If $\int_X e^{\lambda v/n} f^{1/n} (\theta + dd^c \psi)^n = 0$, then 
          $$ \int_{\{v \geq -t\}} f^{1/n} (\theta + dd^c \psi)^n \leq e^{t/n} \int_X e^{\lambda v/n} f^{1/n} (\theta + dd^c \psi)^n = 0, \; \; \forall t \geq 0, $$
          hence $f = 0$ in $L^1(\theta_\psi^n)$, consequently $\mu=0$, a contradiction. We thus infer  that
        $$ \int_X e^{\lambda v_j/n} f_j^{1/n} (\theta + dd^c \psi)^n \geq c' > 0. $$
          Recall that, by \eqref{Mixed MA ineq}, we have 
             $$ (\theta + dd^c v_j) \wedge (\theta + dd^c \psi)^{n-1}  \geq c_j^{1/n} e^{\lambda v_j/n} f_j^{1/n} (\theta + dd^c \psi)^n.  $$
             Integrating over $X$ and using Stokes theorem, we obtain 
             \begin{align*}
                  c' c_j^{1/n} \leq \int_X \theta_{v_j} \wedge \theta_\psi^{n-1} \leq \int_X \theta \wedge \omega_\psi^{n-1} + \int_X v_j dd^c \omega_\psi^{n-1}.
             \end{align*}
          
             Observe that 
             \begin{align*}
                 dd^c \omega_\psi^{n-1} &= (n-1) dd^c \omega \wedge \omega_{\psi}^{n-2} + n(n-1) d\omega \wedge d^c \omega \wedge \omega_{\psi}^{n-3} \\
                 &\geq -n(n-1) B \left(\omega^2 \wedge \omega_\psi^{n-2} + \omega^3 \wedge \omega_\psi^{n-3}\right), 
             \end{align*}
                   where $B\geq 0$ is such that $dd^c \omega \geq -B\omega^2$ and $d\omega \wedge d^c \omega \geq -B\omega^3$. 
             Therefore, we have 
             $$ \int_X v_j dd^c \omega_\psi^{n-1} \leq - n(n-1) B \int_X v_j \left(\omega^2 \wedge \omega_\psi^{n-2} + \omega^3 \wedge \omega_\psi^{n-3}\right), $$
             which is bounded from above by the CLN  inequality \cite[Proposition 1.1]{Ng16AIM}. We thus infer that $u_j$ is uniformly bounded from above on $X$.

       On the other hand, we have 
       $$ e^{-\lambda u_j} \theta_{u_j}^n = f_j \theta_{\psi}^n \leq e^{-\lambda(\psi-C)} \theta_\psi^n  $$
       for some $C>0$. It follows from \cite[Proposition 2.8]{GL22} that $u_j \geq \psi -C$, hence $u_j$ is uniformly bounded. Up to extracting we can assume $u_j \rightarrow u \in \PSH(X,\theta) \cap L^\infty(X)$ in $L^1(X)$. Using Lemma \ref{lem 11.5 GZbook} and extracting if necessary, we can further assume $u_j \rightarrow u$ a.e. with respect $L^1(\theta_\psi^n)$.
      We then obtain by the Lebesgue convergence theorem that $f_j e^{\lambda u_j}\theta_\psi^n \rightarrow f e^{\lambda u}\theta_\psi^n$ strongly.

      On the other hand, using the fact that $u_j$, $f_j$ are uniformly bounded, we get  $\theta_{u_j}^n \leq C \theta_{\psi}^n$ for some $C>0$. It then follows from Lemma \ref{lem 11.5 GZbook} that 
      $$ \lim_{j\rightarrow +\infty} \int_X |e^{u_j} - e^u| \theta_{u_j}^n =0. $$
      Arguing as in the end of the proof of Step 1, we get  $\theta_{u_j}^n \rightarrow \theta_u^n$ weakly, and consequently  $\theta_u^n = e^{\lambda u} \mu$. 
\end{proof}
We will also need the following subsolution theorem. 
\begin{theorem}\label{thm: sol to MA eqt with right hand-side is the sum of MA mesures}
    Let $u$ and  $v \in \PSH(X,\theta)$ with $u \simeq\phi$ and $v \simeq\phi$. If $\mu:=  e^{-\lambda u} \theta_u^n + e^{-\lambda v} \theta_v^n$    is a Radon measure,  then  there is $\varphi \in \PSH(X,\theta)$ such that  $\varphi \simeq\phi$ and  
    $$ \theta_{\varphi}^n = e^{\lambda \varphi} \mu. $$
\end{theorem}
We first prove the following technical lemma. 
\begin{lemma}\label{lem : control of max(u,-j)}
    For every $u \in \PSH(X,\omega)$, there is $C \geq 0$ such that 
    $$ \int_X (\omega + dd^c \max(u,-j))^n \leq Cj^n, \; \; \forall j \geq 1. $$
\end{lemma}
\begin{proof}
The functions $(\max(u/j,-1))$ form a uniformly bounded sequence of $\omega$-psh functions. We can then find a uniform constant $C\geq 0$ such that 
\[ \int_X (\omega +dd^c \max(u/j,-1))^n \leq C. \]
Therefore 
\begin{align*}
    \int_X (\omega + dd^c \max(u,-j))^n &\leq 
\int_X (j\omega + dd^c \max(u,-j))^n \\ &\leq j^n \int_X (\omega +dd^c \max(u/j,-1))^n \leq Cj^n. 
\end{align*}
\end{proof}
We now proceed to the proof of Theorem \ref{thm: sol to MA eqt with right hand-side is the sum of MA mesures}.
\begin{proof}
 For each $j \in \mathbb{N}$, we set $u_j=\max(u,-j)$, $v_j=\max(v,-j)$ and 
 $$ \mu_j  =  e^{-\lambda u_j} \theta_{u_j}^n + e^{-\lambda v_j} \theta_{v_j}^n. $$
The latter is  a positive Radon measure on $X$ that satisfies 
 $$  \mu_j  \leq 2^n e^{\lambda j} \theta_{{(u_j+v_j)/2}}^n. $$
  By Theorem \ref{thm: ext of bounded sol}, there is $\varphi_j \in \PSH(X,\theta)\cap L^\infty(X)$ such that 
   $$ \theta_{\varphi_j}^n = e^{\lambda \varphi_j}\mu_j. $$
  On one hand, \cite[Corollary 2.9]{GL22} yields $\varphi_j \leq  \min(u_j,v_j)$ for every $j$. 
  On the other hand, since $u$ and $v$ have the same singularity type, there is  $C>0$ such that $|u - v| \leq C$. Observe that we also have $|u_j - v_j| \leq C$ 
  by the definition of $u_j$ and $v_j$.   Now consider the function  $\psi_j = (u_j+ v_j)/2 -C/2 - n\log 2/\lambda$. We have 
  $$ e^{-\lambda \psi_j} \theta_{\psi_j}^n \geq e^{-\lambda/2(u_j+v_j) + C\lambda /2} (\theta_{u_j}^n + \theta_{v_j}^n) \geq e^{-\lambda u_j} \theta_{u_j}^n + e^{-\lambda v_j} \theta_{v_j}^n = e^{-\lambda \varphi_j} \theta_{\varphi_j}^n.  $$
  It follows from \cite[Corollary 2.9]{GL22} that $\varphi_j \geq \psi_j$ for every $j$. By the definition of $\psi_j$ and the fact that $u \simeq \phi$ and $v \simeq \phi$, we get $\varphi_j \geq \phi - C_1$ for some uniform constant $C_1$. We finally obtain $\phi - C_1\leq \varphi_j \leq \min(u_j,v_j)$ for every $j$.

  By the Hartogs lemma,  $(\varphi_j)$ has a subsequence converging in $L^1(X)$ to a function $\varphi \in \PSH(X,\theta)$ with $\varphi\simeq \phi$. Passing to a subsequence if necessary, one can suppose $\varphi_j \rightarrow \varphi$ in $L^1(X)$ and almost everywhere. 
 
   We are going to show that $\varphi$ is the desired solution. A key tool to do this is to prove that $\varphi_j \rightarrow \varphi$ in capacity. By Lemma \ref{lem: cvg in capacity}, it suffices to prove  
 $$  \lim_{j\rightarrow +\infty} \int_X |e^{\varphi_j} - e^\varphi| \theta_{\varphi_j}^n = 0. $$
 Recall that $\varphi_j \leq \min(u_j,v_j)$ and $\varphi \leq \min(u,v)$. 
On one hand, we have by Theorem \ref{thm: plurifine property and Demailly type inequality}
 \begin{align}\label{eq: first equation}
     \int_{\{\min(u,v)>-j\}} |e^{\varphi_j} - e^\varphi| \theta_{\varphi_j}^n &= \int_{\{\min(u,v)>-j\}} |e^{\varphi_j} - e^\varphi| e^{\lambda \varphi_j} d\mu \nonumber\\
     &\leq e^{\lambda  \sup_X u}\int_X |e^{\varphi_j} - e^{\varphi}| d\mu,
 \end{align}
  where we have taken $j$ large enough so that $\sup_X u \geq -j$.  On the other hand,
  \begin{equation}\label{eq: 1}
      |e^{\varphi} - e^{\varphi_j}| \leq \max(e^{\varphi}, e^{\varphi_j}) \leq e^{-j} \; \; \text{on} \; \{\min(u,v) \leq -j\}.
  \end{equation}
 Therefore
  $$ \int_{\{\min(u,v) \leq -j\}} |e^{\varphi_j} - e^\varphi| \theta_{\varphi_j}^n \leq e^{-j} \int_X e^{\lambda (\varphi_j - u_j)} \theta_{u_j}^n + e^{\lambda (\varphi_j - v_j)} \theta_{v_j}^n \leq e^{-j} \int_X (\omega_{u_j}^n + \omega_{v_j}^n).  $$
 We now apply the last lemma to obtain 
  $$  \int_{\{\min(u,v) \leq -j\}} |e^{\varphi_j} - e^\varphi| \theta_{\varphi_j}^n \leq A j^n e^{-j},  $$
  where $A$ is uniform constant that does not depend on $j$. 
  Using this and \eqref{eq: first equation}, we get
  \begin{align*}
       \int_X |e^{\varphi_j} - e^\varphi| \theta_{\varphi_j}^n &=  \int_{\{\min(u,v) >-j\}} |e^{\varphi_j} - e^\varphi| \theta_{\varphi_j}^n + \int_{\{\min(u,v) \leq -j\}} |e^{\varphi_j} - e^\varphi| \theta_{\varphi_j}^n \\
       &\leq e^{\lambda  \sup_X u}\int_X |e^{\varphi_j} - e^{\varphi}| d\mu + Aj^n e^{-j}.
  \end{align*}
   for $j$ large enough. Letting $j \rightarrow +\infty$, it follows from Lemma \ref{lem 11.5 GZbook} that
   $$  \lim_{j\rightarrow +\infty} \int_X |e^{\varphi_j} - e^\varphi| \theta_{\varphi_j}^n = 0, $$
  hence  $\varphi_j \rightarrow \varphi$ in capacity by Lemma \ref{lem: cvg in capacity}. 
  
We shall prove now that $\varphi$ is the desired solution. 
Fix $k \in \mathbb{N}$, $\varepsilon>0$ and consider the functions 
$$ f_k^\varepsilon = \frac{\max(u+k,0)}{\max(u+k,0)+\varepsilon} \; \text{and} \; g_k^\varepsilon = \frac{\max(v+k,0)}{\max(v+k,0)+\varepsilon}.  $$
These are positive, bounded and quasi-continuous functions. Moreover, Theorem \ref{thm: plurifine property and Demailly type inequality} gives 
 $$ f_k^\varepsilon g_k^\varepsilon \theta_{\varphi_j}^n = f_k^\varepsilon g_k^\varepsilon e^{\varphi_j} \mu, \; \; \forall j \geq k. $$
 Fixing $\chi \in \mathcal{C}^0(X)$, $\chi\geq 0$ and using  \cite[Lemma 11.5]{GZbook} and \cite[Theorem 2.6]{DDL23}, we  get
\begin{align*}
    \int_X \chi f_k^\varepsilon g_k^\varepsilon \theta_\varphi^n 
    &\leq \liminf_{j\rightarrow +\infty} \int_X \chi f_k^\varepsilon g_k^\varepsilon \theta_{\varphi_j}^n \\
    &= \liminf_{j\rightarrow +\infty} \int_X \chi f_k^\varepsilon g_k^\varepsilon e^{\lambda \varphi_j} d\mu \\
    &= \int_X \chi f_k^\varepsilon g_k^\varepsilon e^{\lambda \varphi} d\mu.
\end{align*}
Letting $\varepsilon \rightarrow 0$ and then $k\rightarrow +\infty$, we get 
$$  \int_X \chi \theta_\varphi^n  \leq \int_X \chi e^{\lambda \varphi} d\mu.   $$
By our choice of the function $\chi$, this means that $\theta_\varphi^n \leq e^{\lambda \varphi} \mu$. 

On the other hand, the function  $\psi := P_\theta(\inf \varphi_j)$ belongs $\PSH(X,\theta)$.  For every $c$, $\varepsilon>0$, we set
$$ h_\varepsilon^c = \frac{\max(\psi + c,0)}{\max(\psi + c,0)+\varepsilon}. $$
Observe that $0\leq h_\varepsilon^c \leq 1$ and $h_\varepsilon^c \neq 0$ only if $\varphi_j>-c$ for every $j$. Moreover, for $\varphi^c=\max(\varphi,-c)$ and $\varphi_j^c=\max(\varphi_j,-c)$, we have 
$$ |\varphi_j^c - \varphi^c| \leq |\varphi_j - \varphi|, \; \; \forall j \geq 1, \forall c>0. $$
Thus, we infer $\varphi_j^c \rightarrow \varphi^c$ in capacity. It follows from \cite[Theorem 4.26]{GZbook} that
$$ \int_X h_\varepsilon^c  \theta_\varphi^n = 
    \int_X h_\varepsilon^c  \theta_{\varphi^c}^n 
    = \lim_{j\rightarrow +\infty} \int_X h_\varepsilon^c   \theta_{\varphi_j^c}^n 
    \geq \limsup_{j\rightarrow +\infty} \int_{\{\min(u,v)>-j\}} h_\varepsilon^c e^{\lambda \varphi_j} d\mu. $$
    Using \eqref{eq: 1}, we obtain 
    $$  \lim_{j\rightarrow +\infty} \int_{\{\min(u,v)\leq -j\}} h_\varepsilon^c e^{\lambda \varphi_j} d\mu = 0, $$
    hence 
    $$  \int_X h_\varepsilon^c  \theta_\varphi^n \geq \limsup_{j\rightarrow +\infty} \int_X h_\varepsilon^c e^{\lambda \varphi_j} d\mu = \int_X h_\varepsilon^c e^{\lambda \varphi} d\mu, $$
where in the last equality we have used \cite[Lemma 11.5]{GZbook}. 
Letting $\varepsilon \rightarrow 0$ and then $c \rightarrow +\infty$, we obtain $\int_X e^{\lambda \varphi} d\mu \leq \int_X \theta_\varphi^n$.  We have shown that   $\theta_\varphi^n \leq e^{\lambda \varphi}\mu$ and $\int_X e^{\lambda \varphi} d\mu \leq \int_X \theta_\varphi^n$, hence $\theta_\varphi^n=e^{\lambda \varphi}\mu$, finishing the proof.
\end{proof}
We are now in position to prove the existence of weak solutions with model singularities when the right hand side has $L^p$-density.
\begin{theorem}\label{thm: sol to MA eqt with right hand-side has Lp density}
Assume $\mu = f \omega^n$ is a positive Radon measure with  $ f \in L^p(X)$, $p>1$. Then, there is a uniquely determined $u \in \PSH(X,\theta)$ such that $u\simeq \phi$ and
    $$ \theta_u^n = e^{\lambda u} \mu. $$
\end{theorem}
To prove the last theorem, we need to establish the following relative $L^\infty$ estimate.
\begin{lemma}\label{lem: uniform estimate}
    Assume $u \in \PSH(X,\theta)$ satisfies $\sup_X u=0$, $u \simeq \phi$ and  $\theta_u^n \leq  f \omega^n$ for some function $f \in L^p(X)$ with $p>1$.   Then $u \geq \phi - C$ for a uniform constant $C>0$ that only depends on $X$, $\phi$ and an upper bound of $\|f\|_p$. 
\end{lemma} 
\begin{proof}
   Fix $\varepsilon \in (0,1)$ such that $\phi_\varepsilon := P_{\varepsilon \theta}(\phi) \not\equiv - \infty$. According to Skoda’s uniform integrability theorem (see \cite[Theorem 8.11]{GZ17}), there exists $\gamma>0$ and $1<q<p$ such that $e^{-\gamma (\phi_\varepsilon + u)}f \in L^q(\omega^n)$. We stress here that $\gamma$ and $\|e^{-\gamma (\phi_\varepsilon + u)}f\|_q$ depend only on $(X,\omega)$, $\phi$ and $\|f\|_p$. By \cite[Theorem 3.4]{GL23Crelle}, one can find $c>0$ and $v \in \PSH(X,(1 - \varepsilon)\theta) \cap L^\infty(X)$ such that $\sup_X v= 0$ and 
$$ ((1 - \varepsilon)\theta + v)^n = c e^{-\gamma (\phi_{\varepsilon}+u)} f \omega^n.  $$
Moreover, by construction of $c$, we know that it is bounded from below by a constant that only depends on $\phi$ and an upper bound of $\|f\|_p$.

 The function defined on $\mathbb{R}_-$ by $\chi(t) = \frac{-1}{t-1}$ is convex, increasing  and satisfies $0 \leq \chi \leq 1$. For 
$$ h:= \phi_\varepsilon + v \; \text{and} \; \varphi:= \phi + \chi(h-\phi) = \phi - \frac{1}{h-\phi-1},$$
we have $h \in \PSH(X,\theta)$, $\varphi \in L^1(X)$ and $\phi \leq \varphi \leq \phi + 1$. We further have 
\begin{align*}
  \theta +  dd^c \varphi &= \theta + dd^c \phi + dd^c(\chi(h-\phi)) \\
    &= \theta + dd^c \phi + \chi'(h-\phi) dd^c(h-\phi) + \chi''(h-\phi) d(h-\phi) \wedge d^c(h-\phi) \\
    &\geq \theta + dd^c \phi + \chi'(h-\phi) dd^c(h-\phi) \\
    &= (1 - \chi'(h-\phi))(\theta + dd^c \phi) + \chi'(h-\phi) (\theta + dd^c h).
\end{align*}
Now using that $\chi'(t) \leq 1$ on $\mathbb{R}_-$, we obtain 
\[\theta + dd^c \varphi \geq \chi'(h-\phi) (\theta + dd^c h) \geq 0. \]
Therefore, $\varphi\in \PSH(X,\theta)$ and 
\[
(\theta +dd^c \varphi)^n \geq \chi'(h-\phi)^n (\theta + dd^c h)^n. 
\]
Using
\[
\theta_h^n \geq ((1-\varepsilon)\theta +dd^c v)^n \geq ce^{-\gamma \phi_{\varepsilon}} e^{-\gamma u} f\omega^n, 
\]
and 
\[
 0\leq \phi-h+1\leq -h +1 = -\phi_{\varepsilon}-v+1\leq -\phi_{\varepsilon}+C, 
\]
we get 
\[
\theta_{\varphi}^n \geq \frac{1}{(\phi-h+1)^{2n}} ce^{-\gamma\phi_{\varepsilon}} e^{-\gamma u} fdV \geq ac e^{-\gamma u} (\theta+dd^c u)^n,
\]
where $a$ is a positive constant that only depends on $n$ and $\gamma$. 
 Therefore, on the set $\{u < \varphi  - 3 +\log(ac)/\gamma\}$, we have  
$$ (\theta + dd^cu)^n \leq e^{\gamma (u+1) -\log(ac)} (\theta + dd^c \varphi)^n \leq e^{-\gamma} (\theta + dd^c \varphi)^n  $$
because $\varphi \leq \phi + 1 \leq 1$. It then follows from the domination principle (Theorem \ref{thm: the dom prn in [phi]}) that $u \geq \varphi -3 + \log(ac)/\gamma$, and consequently that $u \geq \phi -2 + \log(ac)/\gamma$. The proof is complete.
\end{proof}
We will also need the following lemma. 
\begin{lemma}
  Assume $\mu = f\omega^n$ is a positive Radon measure with density 
  $f\in L^p(X)$ for some $p>1$. Then there is $u \in \PSH(X,\theta)$ such that $u \simeq \phi$  and 
$$ \theta_u^n \leq e^{\lambda u} \mu. $$
\end{lemma}
\begin{proof}
By \cite[Theorem 3.4]{GL23Crelle}, there is  $v \in \PSH(X,\theta) \cap L^\infty(X)$ such that 
$$ \theta_v^n = e^{\lambda v} f \omega^n. $$
Set $\varphi= v - \sup_X v$ and $\psi = P_\theta(\varphi,\phi)$. We have $\psi \in \PSH(X,\theta)$, $\psi \simeq \phi$ and 
$$ \theta_{\psi}^n \leq \textit{1}_{\{\psi = \varphi \leq \phi\}} \theta_{\varphi}^n + \textit{1}_{\{\psi = \phi < \varphi\}} \theta_\phi^n $$
by Proposition \ref{prop: the envelope of min of two fncs}. Since $\phi$ is $\theta$-model, we get by Proposition \ref{prop: the MA measure of P theta[u]} that
$$ \theta_\phi^n(\{\psi = \phi < \varphi\}) \leq \theta^n(\{\psi = \phi < \varphi\}\cap \{\phi=0\})= 0.  $$
Therefore, 
$$ \theta_{\psi}^n \leq \textit{1}_{\{\psi = \varphi \leq \phi\}} \theta_{\varphi}^n = \textit{1}_{\{\psi = \varphi \leq \phi\}} e^{\lambda(\varphi+ \sup_Xv)} f \omega^n \leq e^{\lambda(\psi + \sup_X v)} f \omega^n. $$
It then suffices to take $u =\psi + \sup_X v$. 
\end{proof}
We now proceed to the proof of Theorem \ref{thm: sol to MA eqt with right hand-side has Lp density}.
\begin{proof}[Proof of Theorem \ref{thm: sol to MA eqt with right hand-side has Lp density}]
By the last lemma, there is $\varphi\in \PSH(X,\theta)$ such that $\varphi \simeq \phi$ and $e^{-\lambda \varphi}\theta_\varphi^n \leq f\omega^n$.  
    Consider the set
    $$ \mathcal{B} = \{\mu(X) :  \mu = e^{-\lambda u} \theta_u^n, u \in \PSH(X,\theta), u \simeq \phi \; \text{and} \; e^{-\lambda \varphi}\theta_\varphi^n \leq \mu  \leq f\omega^n\}.  $$
 Observe that $\mathcal{B}$ is a  bounded and non-empty subset of $\mathbb{R}_+$. Therefore, there is   $u_j \in \PSH(X,\theta)$ such that $u_j \simeq \phi$, $e^{-\lambda \varphi}\theta_{\varphi}^n \leq e^{-\lambda u_j}\theta_{u_j}^n \leq  f\omega^n$ and 
 $$   \int_X e^{-\lambda u_j}  \theta_{u_j}^n  \nearrow \sup \mathcal{B}. $$
By Corollary \ref{cor: uniqness MA-sol in [phi]}, we have  $u_j \leq \varphi$ for every $j$. 

Fix $\varepsilon \in (0,1)$ such that $P_{\varepsilon\theta}(\phi) \not\equiv -\infty$. By \cite[Theorem 3.4]{GL23Crelle}, there is  $\psi \in \PSH(X,(1-\varepsilon)\theta) \cap L^\infty(X)$ such that 
$$ ((1-\varepsilon)\theta + dd^c \psi)^n = e^{\lambda \psi} f \omega^n. $$
That implies 
\begin{align*}
     e^{-\lambda u_j}(\theta + dd^c u_j)^n &\leq e^{-\lambda \psi}((1-\varepsilon)\theta + dd^c \psi)^n \\ &\leq e^{-\lambda (\psi+P_{\varepsilon\theta}(\phi))}(\theta + dd^c (\psi+P_{\varepsilon\theta}(\phi)))^n.
\end{align*}
It then follows from Corollary \ref{cor: uniqness MA-sol in [phi]} that $u_j \geq \psi + P_{\varepsilon\theta}(\phi)$; in particular $(\sup_X u_j)$ is uniformly bounded.  Up to extracting a subsequence, we can assume $u_j \rightarrow u \in \PSH(X,\theta)$ in $L^1(X)$. Since $u_j \leq \varphi$ for every $j$, we have $u \leq \varphi$ and  $\theta_{u_j}^n \leq e^{\lambda \sup_X \varphi} f\omega^n$ for every $j$. We first obtain by  Lemma \ref{lem: uniform estimate} that $u_j \geq \phi - A$ for a uniform constant $A>0$, hence $u \simeq \phi$. Then, by observing that $u_j$ satisfies \eqref{condition 2}, we conclude from Theorem \ref{main cvg thm} that $u_j \rightarrow u$ in capacity and $\theta_{u_j}^n \longrightarrow \theta_u^n$ weakly.

   Since $e^{-\lambda u_j}\theta_{u_j}^n \leq f\omega^n$ for every $j$, the sequence $(e^{-\lambda u_j}\theta_{u_j}^n)$ has at least one cluster point  $\nu$ for the weak topology. Up to extracting a subsequence, one can assume $e^{-\lambda u_j}\theta_{u_j}^n \rightarrow \nu$ weakly. 
    Since $f \in L^p(X)$ and $p>1$, the equation $\theta_\psi^n = cf\omega^n$ has a solution $(\psi,c) \in \PSH(X,\theta)\cap L^\infty \times \mathbb{R}_+^*$ by  \cite[Theorem 3.4]{GL23Crelle}. Therefore
    $$ \theta_{u_j}^n = e^{\lambda u_j}e^{- \lambda u_j}\theta_{u_j}^n \rightarrow e^{\lambda u} \nu \; \; \text{weakly,} $$
    by \cite[Lemma 2.5]{DDL23}.   We infer that $\theta_u^n = e^{\lambda u} \nu$ and consequently  
    $$ \int_X e^{-\lambda u} \theta_u^n = \nu(X) = \lim_{j\rightarrow +\infty} \int_X e^{-\lambda u_j}\theta_{u_j}^n = \sup \mathcal{B}.  $$

    To complete the proof, we need to show that $\theta_u^n = e^{\lambda u} f\omega^n$. Assume by contradiction that $\theta_u^n \neq e^{\lambda u} f \omega^n$. By the Radon–Nikodym theorem, one can write $\theta_u^n = ge^{\lambda u} \omega^n$, with $g \in L^p(X)$, $g \leq f$ and $\int_X g \omega^n>0$. Then, by the last lemma there is $v \in \PSH(X,\theta)$   such that $v\simeq \phi$ and
    $$ e^{-\lambda v} \theta_{v}^n \leq (f-g) \omega^n = f\omega^n - e^{-\lambda u} \theta_u^n. $$
    Theorem \ref{thm: sol to MA eqt with right hand-side is the sum of MA mesures} then gives $w \in \PSH(X,\theta)$ satisfying $w \simeq \phi$ and 
    $$  \theta_w^n = e^{\lambda(w-u)} \theta_u^n + e^{\lambda(w -v)} \theta_v^n.  $$
    Therefore $e^{-\lambda \varphi} \theta_\varphi^n \leq e^{-\lambda w} \theta_w^n \leq f\omega^n$ and   
    $$\int_X e^{-\lambda w} \theta_w^n = \sup \mathcal{B} + \int_X e^{-v}\theta_v^n. $$  
    Using that $v \simeq \phi$, we obtain
    $$ 0 <  e^{-\lambda \sup_X v} \int_X \theta_v^n \leq \int_X e^{-\lambda v}\theta_v^n,    $$
    hence $\int_X e^{-\lambda w} \theta_w^n >\sup \mathcal{B}$,  contradicting the  definition of $\mathcal{B}$. Therefore $\theta_u^n = e^{\lambda u} f \omega^n$, which finishes the proof.  
\end{proof}
\begin{theorem}\label{sol to MA for lambda=0}
Assume $\mu= f\omega^n$ is a positive Radon measure with density $f \in L^p(X)$, $p>1$. Then, there is a unique $c>0$ and a function $u \in \PSH(X,\theta)$ with the same singularity type as $\phi$ such that 
$$ (\theta + dd^c u)^n = c\mu. $$ 
\end{theorem}
\begin{proof}
According to the last theorem, for every $j \in \mathbb{N}^*$, there is $u_j \in \PSH(X,\theta)$  such that $u_j \simeq \phi$ and 
$$ \theta_{u_j}^n = e^{u_j/j} \mu. $$
The functions $v_j := u_j - \sup_X u_j$ satisfy $\sup_X v_j = 0$ for every $j$.  Hence, up to extracting a subsequence, one can assume $v_j \rightarrow v \in \PSH(X,\theta)$ in $L^1(X)$. By Lemma \ref{lem 11.5 GZbook}, we can also assume $e^{v_j/j} \to 1$ in $L^1(\mu)$. Set $c_j := e^{\sup_X u_j/j}$. By \cite[Theorem 3.4]{GL23Crelle}, there is $\alpha>0$ and $h \in \PSH(X,\theta) \cap L^\infty(X)$ such that 
$$ (\theta + dd^c h)^n = \alpha \mu. $$
That implies 
$$ e^{-u_j/j} \theta_{u_j}^n = \mu \geq e^{ -(h + \|h\|_\infty + j\log \alpha)/j} \theta_{h}^n. $$
It then follows from Theorem \ref{thm: the dom prn in [phi]} that $u_j \leq h + \|h\|_\infty + j\log \alpha$, which proves that $(c_j)$ is uniformly bounded from above. Therefore, up to extracting, we can assume $c_j \rightarrow c \geq 0$.
That implies  $$ \theta_{v_j}^n = c_j e^{v_j/j} \mu \rightarrow c\mu \; \; \text{strongly.} $$
  On the other hand, we have $\theta_{v_j}^n \leq c_j \mu \leq C\mu$ for some positive constant $C$.  We first obtain by Lemma \ref{lem: uniform estimate} that  $v_j \geq \phi - A$ for some uniform constant $A>0$, thus $v \geq \phi - A$.  Secondly,  $v_j$ satisfies \eqref{condition 2} and converges in $L^1(X)$ to $v$, hence $\theta_{v_j}^n \rightarrow \theta_v^n$ weakly by Theorem \ref{main cvg thm},  which implies $\theta_v^n = c\mu$ and thus $c>0$.

    The uniqueness of the constant $c$ follows from Corollary \ref{cor: uniqness MA-sol in [phi]}. 
\end{proof}
\subsection{The case of non-pluripolar measures}
In what follows we will make use of the hypothesis $v_-(\theta)>0$.
\begin{theorem}\label{main thm C}
    Fix $\lambda>0$ and let $\mu$ be a positive Radon measure vanishing on pluripolar sets. Then there is a unique function $u \in \mathcal{E}(X,\theta,\phi)$ solving the equation
    $$ (\theta + dd^c u)^n = e^{\lambda u}\mu. $$
    Moreover, we have $\int_X \theta_u^n \geq v_{-,\phi}(\theta)$.
\end{theorem}
\begin{proof}
By \cite[Theorem 4.1]{ALS24}, we can write  $\mu = f(\omega + dd^c \psi)^n$ where $\psi \in \PSH(X,\omega)\cap L^\infty(X)$ and $f\in L^1((\omega + dd^c \psi)^n)$. The proof will be divided in three steps.

    {\bf Step 1.} We first assume $f \in \mathcal{C}^0(X)$, $f>0$. Let $\psi_j \in \PSH(X,\omega)\cap \mathcal{C}^\infty(X)$, $\psi_j \searrow \psi$. According to Theorem \ref{thm: sol to MA eqt with right hand-side has Lp density}, for every $j\geq 1$, there is $u_j \in \PSH(X,\theta)$ such that $u_j \simeq \phi$ and 
    $$ \theta_{u_j}^n = e^{\lambda u_j} f (\omega + dd^c \psi_j)^n. $$
     Observe that $e^{-\lambda u_j}\omega_{u_j}^n \geq f\omega_{\psi_j}^n$, $f>0$ and $(\psi_j)$ uniformly bounded. Corollary \ref{cor: uniqness MA-sol in [phi]} thus implies $u_j \leq \psi_j+ C$ for some $C>0$.
   Set $v_j = u_j - \sup_X u_j$ and $c_j = e^{\lambda \sup_X u_j}$. Up to extracting a subsequence, we can assume $v_j \rightarrow v$ in $L^1(X)$ and $c_j \rightarrow c\geq 0$. Observe that 
   $$ \theta_{v_j}^n = c_je^{\lambda v_j} f (\omega + dd^c \psi_j)^n. $$
     Hence  $v_j$  satisfies \eqref{condition 1}. We thus infer from Lemma \ref{lem: P_theta(inf u_j) in E} that $v$, $P_\theta(\inf v_j) \in \mathcal{E}(X,\theta,\phi)$.  Theorem \ref{main cvg thm}  then implies $v_j \rightarrow v$ in capacity and $\theta_{v_j}^n \rightarrow \theta_v^n$ weakly.  In particular, $\int_X \theta_v^n \geq v_{-,\phi}(\theta)$ and  
$$  c_je^{\lambda v_j} f \omega_{\psi_j}^n \rightarrow 
  ce^{\lambda v} f \omega_{\psi}^n \; \; \text{weakly,}   $$
  by \cite[Theorem 4.26]{GZbook}.
 It follows that $c>0$ and $\theta_v^n = ce^{\lambda v} f\omega_\psi^n$. Hence, it suffices to take $u = v+  \lambda^{-1}\log c$.

  {\bf Step 2.} Assume now $f \in L^\infty(X)$. Let $f_j$ be a uniformly bounded sequence of positive continuous functions converging to $f$ a.e. with respect to  $\omega_\psi^n$. Using Egorov's theorem, we can construct a function $h \in L^\infty(X)$ such that $\int_X h \omega_\psi^n > 0$ and $0\leq h \leq f_j$ for every $j$. By Step 1, there is $u_j \in \mathcal{E}(X,\theta,\phi)$ such that  $\int_X \theta_{u_j}^n \geq v_{-,\phi}(\theta)$ and 
  $$ (\theta + dd^c u_j)^n = e^{\lambda u_j} f_j (\omega + dd^c \psi)^n. $$
  Theorem \ref{thm: ext of bounded sol} also gives $\varphi \in \PSH(X,\theta) \cap L^\infty(X)$ such that 
  $$ (\theta + dd^c \varphi)^n = e^{\lambda \varphi} h (\omega + dd^c \psi)^n. $$
  In particular, for every $j\geq 1$, we have 
  $$ e^{-\lambda u_j} (\theta + dd^c {u_j})^n \geq e^{-\lambda \varphi} (\theta + dd^c \varphi)^n, $$
  hence $u_j \leq \varphi$ by Corollary \ref{cor: uniqness MA-sol in [phi]}.
  For every $j$, we set $v_j = u_j - \sup_X u_j$ and $c_j = e^{\lambda \sup_X u_j}$. Passing to a subsequence if necessary, we can assume $c_j \rightarrow c\geq 0$ and $v_j \rightarrow v \in\PSH(X,\theta)$ in $L^1(X)$. Applying Lemma \ref{lem 11.5 GZbook} to the sequence $(e^{v_j})$ and extracting, we can further assume  $v_j \rightarrow v$ a.e. with respect to $\omega_\psi^n$. In particular, we have 
  $$ \theta_{v_j}^n = c_j e^{\lambda v_j} f_j \omega_\psi^n \longrightarrow  c e^{\lambda v} f \omega_\psi^n,  $$
  strongly by the Lebesgue convergence theorem. Since $f_j$ is uniformly bounded, the sequence $(v_j)$ satisfies \eqref{condition 2}, hence $v$ and $P_\theta(\inf v_j)$ are in $\mathcal{E}(X,\theta,\phi)$ by Lemma \ref{lem: P_theta(inf u_j) in E}. It follows from Theorem \ref{main cvg thm} that $\theta_{v_j}^n \rightarrow \theta_v^n$ weakly. Therefore, we have $c>0$, $\int_X \theta_{v}^n \geq v_{-,\phi}(\theta)$ and $\theta_v^n = c e^{\lambda v} f\omega^n$, hence we can take $u= v+ \lambda^{-1}\log c$.

   {\bf Step 3.} We move on to the general case when $f \in L^1(\omega_\psi^n)$.
  By the last step,  for each $j \in \mathbb{N}^*$,  there is $u_j \in \mathcal{E}(X,\theta,\phi)$ such that $\int_X \theta_{u_j}^n \geq v_{-,\phi}(\theta)$ and
    $$ (\theta + dd^c u_j)^n = e^{\lambda u_j} \min(f,j) (\omega + dd^c \psi)^n. $$
    The sequence $(u_j)$ is decreasing by Corollary \ref{cor: uniqness MA-sol in E(X,theta,phi)}; set $u = \lim u_j$. According to Proposition \ref{prop: cond ensure v- positive}, we know that $\inf_j \int_X \theta_{u_j}^n>0$.  Using the fact that
      \[  \int_X \theta_{u_j}^n \leq e^{\lambda \sup_X u_j} \mu(X), \; \; \forall j \geq 1, \]
   we obtain  $u \in \PSH(X,\theta)$. Observe also that $u_j$ satisfies the condition \eqref{condition 2} since $\theta_{u_j}^n \leq e^{\lambda u_1} \mu$ for every $j$. It then follows from Lemma \ref{lem: P_theta(inf u_j) in E} that $u \in \mathcal{E}(X,\theta,\phi)$ and from Theorem \ref{main cvg thm} that $\theta_{u_j}^n \rightarrow \theta_u^n$ weakly. In particular, $\int_X \theta_{u}^n \geq v_{-,\phi}(\theta)$. On the other hand,  the Lebesgue convergence theorem implies  
  $$ \theta_{u_j}^n = e^{\lambda u_j} \min(f,j) (\omega + dd^c \psi)^n \rightarrow e^{\lambda u} f (\omega + dd^c \psi)^n, \; \; \text{strongly.} $$
  Therefore $\theta_u^n = e^{\lambda u} \mu$.

  The uniqueness of the solution follows from Corollary \ref{cor: uniqness MA-sol in E(X,theta,phi)}.
\end{proof}
\begin{corollary}\label{cor: sol of MA equation}
 Let $\mu$ be a positive Radon measure vanishing on pluripolar sets. Then, there is a unique $c>0$ and a function $u \in \mathcal{E}(X,\theta,\phi)$ such that $\int_X \theta_{u}^n \geq v_{-,\phi}(\theta)$ and
    $$ (\theta + dd^c u)^n = c\mu. $$    
\end{corollary}
\begin{proof}
By the last theorem, for every $j \in \mathbb{N}^*$, we can find $u_j \in \mathcal{E}(X,\theta,\phi)$  such that 
$$ (\theta + dd^c {u_j})^n = e^{u_j/j} \mu. $$
According to \cite[Theorem 4.1]{ALS24}, we can write  $\mu = f(\omega + dd^c \psi)^n$ where $\psi \in \PSH(X,\omega)\cap L^\infty(X)$ and $f\in L^1((\omega + dd^c \psi)^n)$. Therefore, Theorem \ref{thm: ext of bounded sol} gives $h \in \PSH(X,\theta) \cap L^\infty(X)$ such that 
$$ (\theta + dd^c h)^n = e^h \min(f,1) (\omega + dd^c \psi)^n. $$
That implies 
$$ e^{-u_j/j} (\theta + dd^c {u_j})^n \geq e^{-h} (\theta + dd^c h)^n, $$ 
and hence
$$ \theta_h^n \leq e^{h - u_j/j} \theta_{u_j}^n \leq e^{(h + (j-1) \|h\|_\infty -u_j)/j} \theta_{u_j}^n.   $$
It then follows from Corollary \ref{cor: uniqness MA-sol in [phi]} that $u_j \leq j \|h\|_\infty$. In particular, the sequence $c_j := e^{\sup_X u_j/j}$ is uniformly bounded. Also, the functions $v_j := u_j - \sup_X u_j$ are $\theta$-psh and satisfy $\sup_X v_j = 0$ for every $j$.  Hence, up to extracting a subsequence, one can assume $c_j \rightarrow c \geq 0$ and $v_j \rightarrow v \in \PSH(X,\theta)$ in $L^1(X)$. By Lemma \ref{lem 11.5 GZbook}, we can also assume $e^{v_j/j} \to 1$ in $L^1(\mu)$. Therefore
$$ \theta_{v_j}^n = c_j e^{v_j/j} \mu \rightarrow c\mu \; \; \text{strongly.} $$
  On the other hand, we have $\theta_{v_j}^n \leq c_j \mu \leq C\mu$ for some positive constant $C$. This means that  $v_j$ satisfies \eqref{condition 2}. We first obtain by Lemma \ref{lem: P_theta(inf u_j) in E} that $v$ and $P_\theta(\inf_j v_j)$ are in $\mathcal{E}(X,\theta,\phi)$. Theorem \ref{main cvg thm} then implies  $\theta_{v_j}^n \rightarrow \theta_v^n$ weakly. We thus infer $\theta_v^n = c\mu$ and $c>0$.

    The uniqueness of the constant $c$ follows from Corollary \ref{cor: uniqness MA-sol in E(X,theta,phi)}. 
\end{proof}
\begin{corollary}\label{cor : lower bound of fcts in E}
 For every $u \in \mathcal{E}(X,\theta,\phi)$, we have $v_{-,u}(\theta) = v_{-,\phi}(\theta)$.
\end{corollary}
\begin{proof}
Fix $u \in  \mathcal{E}(X,\theta,\phi)$. According to Proposition \ref{prop: comparison of MA vol}, we have  $v_{-,u}(\theta) \leq v_{-,\phi}(\theta)$. Let
$v \in \PSH(X,\theta)$, $u \simeq v$ be such that $\int_X \theta_v^n<+\infty$.
By Corollary \ref{cor: sol of MA equation}, there is $c>0$ and $w \in \mathcal{E}(X,\theta,\phi)$ such that $\int_X \theta_w^n \geq v_{-,\phi}(\theta)$ and
$$ (\theta + dd^c w)^n = c(\theta + dd^c v)^n. $$
 According to  Corollary \ref{cor: uniqness MA-sol in E(X,theta,phi)}, we have $c=1$, hence $\int_X\theta_v^n \geq v_{-,\phi}(\theta)$. Taking infimum over all functions $v$ yields $v_{-,u}(\theta) \geq v_{-,\phi}(\theta)$, whence equality. 
\end{proof}
We now provide a characterization of the relative full mass in terms of the envelope of singularity types, generalizing \cite[Theorem 2.14]{ALS24}. 
\begin{theorem}\label{thm: Ptheta[u]=phi}
 Let $u \in \PSH(X,\theta)$ be such that $\int_X \theta_u^n<+\infty$. Then 
$u \in \mathcal{E}(X,\theta,\phi)$ if and only if $P_\theta[u] = \phi$.
\end{theorem}
The proof of the last theorem is based on the following lemma.
\begin{lemma}
    For every $u \in \PSH(X,\theta)$, we have $v_{-,u}(\theta) = v_{-,P_\theta[u]}(\theta)$.
\end{lemma}
\begin{proof}
By Proposition \ref{prop: comparison of MA vol}, we always have $v_{-,u}(\theta) \leq v_{-,P_\theta[u]}(\theta)$. 

We first assume $v_{-,u}(\theta)>0$. On one hand, we get by Proposition \ref{prop: cons of model pot} that $P_\theta[u]$ is a $\theta$-model potential. On the other hand,  Proposition \ref{u in E(X,theta,Ptheta u)} implies $u \in \mathcal{E}(X,\theta,P_\theta[u])$, thus $v_{-,u}(\theta) = v_{-,P_\theta[u]}(\theta)$ by Corollary \ref{cor : lower bound of fcts in E}.

Assume now $v_{-,u}(\theta)=0$. Fix $\delta>0$ and set $\theta_\delta = (1+\delta)\theta$. Since $v_{-}(\theta)>0$ by assumption, Proposition \ref{prop: controling mass} implies  $v_{-,u}(\theta_\delta)>0$. It follows from the above that $v_{-,u}(\theta_\delta) = v_{-,P_{\theta_\delta}[u]}(\theta_\delta)$. Since $u \preceq P_\theta[u]$ and   $P_\theta[u] \leq P_{\theta_\delta}[u]$, we get by Proposition \ref{prop: comparison of MA vol} that $v_{-,P_\theta[u]}(\theta) \leq v_{-,u}(\theta_\delta)$. That implies
$$ v_{-,P_\theta[u]}(\theta) \leq \int_X (\theta_\delta + dd^c v)^n,  $$
for every $v \in \PSH(X,\theta)$ with $v \simeq u$. Therefore, letting $\delta \rightarrow 0$ and then taking the infimum over all functions $v$, we obtain $v_{-,P_\theta[u]}(\theta) = 0$. The proof is complete. 
\end{proof}
We are now ready to prove Theorem \ref{thm: Ptheta[u]=phi}.
\begin{proof}[Proof of Theorem \ref{thm: Ptheta[u]=phi}]
  Assume $u \in \mathcal{E}(X,\theta,\phi)$. Since $P_\theta[u]= P_\theta[u+c]$ for every $c\in \mathbb{R}$, there is no loss of generality in assuming $u\leq 0$. In particular, we have $u \leq \phi$ and $P_\theta[u] \leq P_\theta[\phi]= \phi$. Since we also have $u \leq P_\theta[u]$, it follows that $P_\theta[u] \in \mathcal{E}(X,\theta,\phi)$.
   Moreover, according to  Proposition \ref{prop: the MA measure of P theta[u]}, the measure $\theta_{P_\theta[u]}^n$ is carried by $\{P_\theta[u]=0\}$, hence   $\theta_{P_\theta[u]}^n = 0$ on $\{P_\theta[u]<\phi]$.  It then follows from Theorem \ref{thm: the dom prn in E(X,theta,phi)} that $P_\theta[u] \geq\phi$, and consequently $P_\theta[u] =\phi$.

   Assume now $u \in \PSH(X,\theta)$ satisfies $u \preceq \phi$ and $P_\theta[u] = \phi$. By the last lemma, we have $v_{-,u}(\theta)>0$, hence $u \in \mathcal{E}(X,\theta,\phi)$ by  Proposition \ref{u in E(X,theta,Ptheta u)}. The proof is complete. 
\end{proof}
\begin{remark}
Using Theorem \ref{thm: Ptheta[u]=phi}, one can argue as in the proof of \cite[Theorem 1.1]{DDL1} to show that $\nu(u,x)= \nu(\phi,x)$ for every $x \in X$ and every  $u \in \mathcal{E}(X,\theta,\phi)$ satisfying  $\int_X \theta_u^n<+\infty$. Here $\nu(u,x)$ denotes  the Lelong number of $u$ at the point $x$:
$$ \nu(u,x) = \sup \{\gamma \geq 0: u(z) \leq \gamma \log\|x-z\| + O(1) \; \text{on }  U \} $$
where $U$ is a local holomorphic chart in $X$ containing the point $x$. In particular, for $\phi=0$, this means that functions in $\mathcal{E}(X,\theta)$ have zero Lelong number at all points.       
\end{remark}
\newcommand{\etalchar}[1]{$^{#1}$}
\providecommand{\bysame}{\leavevmode\hbox to3em{\hrulefill}\thinspace}
\providecommand{\MR}{\relax\ifhmode\unskip\space\fi MR }
\providecommand{\MRhref}[2]{%
  \href{http://www.ams.org/mathscinet-getitem?mr=#1}{#2}
}
\providecommand{\href}[2]{#2}

\end{document}